\documentclass[a4paper,reqno,tbtags,draft=false]{amsart}

\newtheorem{theorem}{Theorem}[section]
\newtheorem{lemma}[theorem]{Lemma}
\newtheorem{corollary}[theorem]{Corollary}
\newtheorem{proposition}[theorem]{Proposition}

\theoremstyle{definition}
\newtheorem{definition}[theorem]{Definition}

\numberwithin{equation}{section}

\usepackage[utf8]{inputenc}
\usepackage[T1]{fontenc}

\usepackage{booktabs}
\usepackage[headings]{fullpage}
\usepackage{paralist} 

\usepackage{amssymb,mathtools}
\usepackage{mathdots}
\usepackage{bbold}

\newcommand\uprightlozenge[1][1]{\begin{tikzpicture}[scale=#1]
	\draw (0,0) --++(60:1) --++(-60:1) --++(-120:1) --++(120:1);
	\end{tikzpicture}}

\DeclareMathOperator{\sgn}{sgn}
\DeclareMathOperator{\id}{id}
\DeclareMathOperator{\E}{E}
\DeclareMathOperator{\fd}{\Delta}
\DeclareMathOperator{\bd}{\delta}
\DeclareMathOperator{\HMT}{HMT}
\DeclareMathOperator{\ct}{CT}
\DeclareMathOperator{\asym}{\mathbf{ASym}}
\DeclareMathOperator{\sym}{\mathbf{Sym}}
\DeclareMathOperator{\Strict}{Strict}

\DeclareMathOperator{\Qid}{{}^\textit{Q}id}
\DeclareMathOperator{\QE}{{}^\textit{Q}E}
\DeclareMathOperator{\Qfd}{{}^\textit{Q}\Delta}
\DeclareMathOperator{\QStrict}{{}^\textit{Q}Strict}
\DeclareMathOperator{\T}{{}^\textit{Q}\overline{Strict}}

\makeatletter
\DeclareRobustCommand\bigop[1]{%
	\mathop{\vphantom{\sum}\mathpalette\bigop@{#1}}\slimits@
}
\newcommand{\bigop@}[2]{%
	\vcenter{%
		\sbox\z@{$#1\sum$}%
		\hbox{\resizebox{\ifx#1\displaystyle.9\fi\dimexpr\ht\z@+\dp\z@}{!}{$\m@th#2$}}%
	}%
}
\makeatother

\newcommand{\bbsum}{\DOTSB\bigop{\mathbb{\Sigma}}}
	
\newcommand{\Qsum}[2]{\sum\limits_{(l_1,\dots ,l_{#1})}^{(k_1,\dots ,k_{#2})}}
\newcommand{\AQsum}[2]{\bbsum\limits_{(l_1,\dots ,l_{#1})}^{(k_1,\dots ,k_{#2})}}
\newcommand{\AQbsum}[1]{\bbsum\limits_{(l_1,\dots ,l_{#1})}^{(k_1,\dots ,k_{#1},b)}}

\newcommand{\QHMT}{{}^Q\HMT}

\newcommand{\lle}{\rotatebox[origin=c]{45}{$\le$}}
\newcommand{\rle}{\rotatebox[origin=c]{-45}{$\le$}}

\usepackage{graphicx}
\usepackage{tikz}
\usetikzlibrary{calc,decorations.pathreplacing}
\usepackage{pgffor}

\newcounter{xhmt}
\newcounter{yhmt}

\newcommand\hmt[1]{
	\setcounter{yhmt}{-1}
	\foreach \p in {#1} {
		\addtocounter{yhmt}{1}
		\setcounter{xhmt}{\value{yhmt}-2}
		\foreach \q in \p {
			\addtocounter{xhmt}{2}
			\node at (\value{xhmt},\value{yhmt}) {\q};   
		}
	}
}

\usepackage[bookmarks,hyperindex,colorlinks,pdftex]{hyperref}

\begin{document}

\title[Refined Enumeration of Halved Monotone Triangles]{Refined Enumeration of Halved Monotone Triangles and Applications to Vertically Symmetric Alternating Sign Trapezoids}

\author{Hans Höngesberg}
\address{Universität Wien, Fakultät für Mathematik, Oskar-Morgenstern-Platz 1, 1090 Wien, Austria}
\email{hans.hoengesberg@univie.ac.at}


\date{}

\dedicatory{}

\thanks{Supported by the Austrian Science Fund FWF, START grant Y463 and SFB grant F50.}

\begin{abstract}
	Halved monotone triangles are a generalisation of vertically symmetric alternating sign matrices (VSASMs). We provide a weighted enumeration of halved monotone triangles with respect to a parameter which generalises the number of $-1$s in a VSASM. Among other things, this enables us to establish a generating function for vertically symmetric alternating sign trapezoids. Our results are mainly presented in terms of constant term expressions. For the proofs, we exploit Fischer's method of operator formulae as a key tool.
\end{abstract}

\maketitle

\section{Introduction}
	After Robbins and Rumsey had introduced alternating sign matrices \cite{RR86} and conjectured together with Mills an enumeration formula for alternating sign matrices of size $n \times n$ \cite[Conjecture 1]{MRR82}, it took more than ten years to prove this formula. Zeilberger presented the first proof \cite{Zei96a}; the key ingredients of his intricate proof are constant term identities with the aid of which he shows that alternating sign matrices are equinumerous with totally symmetric, self-complementary plane partitions. Shortly thereafter, Kuperberg provided a second, much shorter and compact proof \cite{Kup96} exploiting methods from statistical mechanics via the six-vertex model. The latter approach has become the prominent tool in studying alternating sign arrays, especially symmetry classes of alternating sign matrices as well as other interesting subclasses. The systematic study of symmetric alternating sign matrices, initiated by Stanley \cite{Rob91}, has become a fruitful but arduous task: Robbins initially conjectured a list of simple enumeration formulae \cite{Rob}, the last one thereof has only recently been solved \cite{BFK17}. 
	
	Alternating sign triangles are a newly introduced class of alternating sign arrays. Ayyer, Behrend, and Fischer have shown that alternating sign triangles with $n$ rows and $n \times n$ alternating sign matrices are equinumerous \cite{ABF}. Moreover, it is conjectured that even the generating functions of vertically symmetric alternating sign triangles and vertically symmetric alternating sign matrices with respect to the number of $-1$s are equal \cite[p.~33]{ABF}.
	
	Alternating sign triangles have been generalised to alternating sign trapezoids.
	The notion of $(n,l)$-alternating sign trapezoids with bases of odd length~$l$ has first been introduced by Ayyer \cite{Ayy} and by Aigner \cite{Aig17}; Behrend and Fischer \cite{BF} expanded the notion to include alternating sign trapezoids with bases of even length~$l$. In this way, alternating sign trapezoids generalise alternating sign triangles and quasi alternating sign triangles as defined in \cite{ABF} at the same time. It has been independently conjectured, first by Behrend and later by Aigner, that $(n,l)$-alternating sign trapezoids are equinumerous with column strict shifted plane partitions of class~$l-1$ with at most $n$ parts in the top row. This fact is shown by Fischer \cite{Fis19Trap} by means of operator formulae and constant term expressions. In addition, Behrend and Fischer will present a proof using the six-vertex model in a forthcoming paper \cite{BF}.
	
	The basic objects of our investigation are halved monotone triangles which are originally defined and enumerated in \cite{Fis09}. We introduce halved trees as a generalisation of halved monotone triangles. The purpose of this paper is to provide refined enumeration formulae for halved monotone triangles and halved trees in terms of operator formulae and constant term expressions to study vertically symmetric alternating sign trapezoids. In Section~\ref{sec:Preliminaries}, we provide the basic definitions and explain the correspondence between halved trees and alternating sign triangles and trapezoids. In Section~\ref{sec:EnumHMT}, we discuss the refined enumeration of halved monotone triangles and halved trees. Theorems~\ref{thm:QHTREEenumeration} and \ref{thm:QHTREEConstantTerm} state the main enumeration formulae for halved trees, which are applied to the case of alternating sign arrays in Section~\ref{sec:EnumAST}. Theorem~\ref{thm:VSASTPQEnumeration} and Corollary~\ref{thm:VSASTPQEnumerationOdd} establish generating functions of vertically symmetric alternating sign trapezoids. In Section~\ref{sec:Proofs}, we provide the proofs and some technical details. Finally, we consider the $2$-enumeration of halved monotone triangles in Appendix~\ref{sec:2enum} and make some remarks about the enumeration of the closely related halved Gelfand-Tsetlin patterns in Appendix~\ref{sec:HGTP}.
	
	Since we provide the first expression so far for the number of vertically symmetric alternating sign triangles, our results and especially Corollary~\ref{thm:VSASTPQEnumeration} are a step towards the proof of the conjecture that there are as many vertically symmetric alternating sign triangles with $n$ rows as vertically symmetric $n \times n$ alternating sign matrices.
	
\section{Preliminaries}
\label{sec:Preliminaries}
In this section, we define the basic objects of our study and explain the relation between vertically symmetric alternating sign trapezoids and halved trees. This correspondence is crucial for the enumerations in Section~\ref{sec:EnumAST}.
 
\begin{definition}
	\label{def:ASTrap}
For given integers~$n \geq 1$ and $l \geq 2$, an \emph{$(n,l)$-alternating sign trapezoid} is defined as an array of integers in a trapezoidal shape with entries~$-1$, $0$, or $+1$ arranged in $n$ centred rows of length~$2n+l-2, 2n+l-4,\dots,l+2$, and $l$ in the following way
\begin{equation*}
	\begin{array}[t]{ccccccccc}
		a_{1,1}&a_{1,2}&\cdots&\cdots&\cdots&\cdots&\cdots&\cdots&a_{1,2n+l-2}\\
		&a_{2,2}&\cdots&\cdots&\cdots&\cdots&\cdots&a_{2,2n+l-3}&\\
		&&\ddots&&&&\reflectbox{$\ddots$}&&\\
		&&&a_{n,n}&\cdots&a_{n,n+l-1}&&&
	\end{array}
\end{equation*}
such that
\begin{compactitem}
	\item the nonzero entries alternate in sign in each row and each column,
	\item the topmost nonzero entry in each column is $1$ (if it exists),
	\item the entries in each row sum up to $1$, and
	\item the entries in the central $l-2$ columns sum up to $0$.
\end{compactitem}
In the case of $l=1$, an \emph{$(n,1)$-alternating sign trapezoid} is defined as above with the exception that the entry in the bottom row can be $0$ or $1$.
\end{definition}

For instance, there are nine $(2,5)$-alternating sign trapezoids, which are provided in Table~\ref{fig:(2,5)-ASTrap}.
\begin{table}[htb]
	\centering
	\caption{List of all $(2,5)$-alternating sign trapezoids}
\begin{tabular}{ccccc}
\toprule
\begin{array}[t]{ccccccc}
$1$  &  $0$  &  $0$  &  $0$  &  $0$  &  $0$  &  $0$\\
     &  $1$  &  $0$  &  $0$  &  $0$  &  $0$  &
\end{array}

& &

\begin{array}[t]{ccccccc}
$0$  &  $0$  &  $0$  &  $0$  &  $0$  &  $1$  &  $0$\\
     &  $1$  &  $0$  &  $0$  &  $0$  &  $0$  &
\end{array}

& &

\begin{array}[t]{ccccccc}
$0$  &  $0$  &  $0$  &  $0$  &  $0$  &  $0$  &  $1$\\
     &  $1$  &  $0$  &  $0$  &  $0$  &  $0$  &
\end{array}

\\\midrule

\begin{array}[t]{ccccccc}
$1$  &  $0$  &  $0$  &  $0$  &  $0$  &  $0$  &  $0$\\
     &  $0$  &  $0$  &  $0$  &  $0$  &  $1$  &
\end{array}

& &

\begin{array}[t]{ccccccc}
$0$  &  $1$  &  $0$  &  $0$  &  $0$  &  $0$  &  $0$\\
     &  $0$  &  $0$  &  $0$  &  $0$  &  $1$  &
\end{array}

& &

\begin{array}[t]{ccccccc}
$0$  &  $0$  &  $0$  &  $0$  &  $0$  &  $0$  &  $1$\\
     &  $0$  &  $0$  &  $0$  &  $0$  &  $1$  &
\end{array}

\\\midrule

\begin{array}[t]{ccccccc}
$0$  &  $0$  &  $1$   &  $0$  &  $0$  &  $0$  &  $0$\\
     &  $1$  &  $-1$  &  $0$  &  $0$  &  $1$  &
\end{array}

& &

\begin{array}[t]{ccccccc}
$0$  &  $0$  &  $0$  &  $1$   &  $0$  &  $0$  &  $0$\\
     &  $1$  &  $0$  &  $-1$  &  $0$  &  $1$  &
\end{array}

& &

\begin{array}[t]{ccccccc}
$0$  &  $0$  &  $0$  &  $0$  &  $1$   &  $0$  &  $0$\\
     &  $1$  &  $0$  &  $0$  &  $-1$  &  $1$  &
\end{array}\\
\bottomrule
\end{tabular}
	\label{fig:(2,5)-ASTrap}
\end{table}

Note that alternating sign triangles of order~$n$ correspond to $(n-1,3)$-alternating sign trapezoids by deleting the bottom row and that quasi alternating sign triangles coincide with $(n,1)$-alternating sign trapezoids.

The entries in each column of an alternating sign trapezoid sum up to $0$ or $1$ by the first and the second condition in the Definition~\ref{def:ASTrap}. If the column sum is $1$, we call this column a \emph{$1$-column}, otherwise a \emph{$0$-column}. Among the $1$-columns we distinguish the \emph{$10$-columns} from the \emph{$11$-columns} depending on whether the bottom entry is $0$ or $1$, respectively.

We focus on \emph{vertically symmetric alternating sign trapezoids}. These are alternating sign trapezoids which stay invariant under reflection along a vertical axis of symmetry. Since the nonzero entries in each row alternate in sign and sum up to $1$, there is exactly one $1$ and no other nonzero entry in the top row by the second condition. Therefore, due to symmetry, there exists a central column. Hence, $l$ has to be odd. Figure~\ref{fig:VSASM} shows an example of a vertically symmetric $(6,9)$-alternating sign trapezoid.

\begin{figure}[htb]
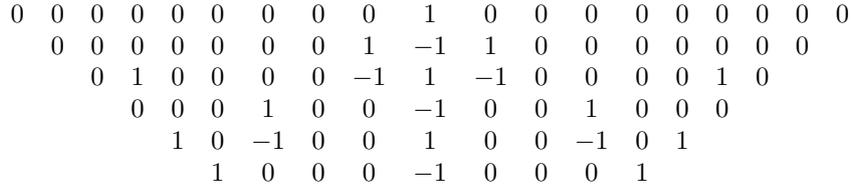

	\centering
\begin{equation*}
	\begin{array}[t]{ccccccccccccccccccc}
		0&0&0&0&0&0&0&0&0&1&0&0&0&0&0&0&0&0&0\\
		&0&0&0&0&0&0&0&1&-1&1&0&0&0&0&0&0&0&\\
		&&0&1&0&0&0&0&-1&1&-1&0&0&0&0&1&0&&\\
		&&&0&0&0&1&0&0&-1&0&0&1&0&0&0&&&\\
		&&&&1&0&-1&0&0&1&0&0&-1&0&1&&&&\\
		&&&&&1&0&0&0&-1&0&0&0&1&&&&&\\
	\end{array}
\end{equation*}
\caption{Vertically symmetric $(6,9)$-alternating sign trapezoid}
\label{fig:VSASM}
\end{figure}

If we number the $n$ leftmost columns of an ($n$,$l$)-alternating sign trapezoid from $-n$ to $-1$, we can associate the \emph{$1$-column vector} $\mathbf{c}=\left(c_1,\dots,c_m\right)$ with $-n \le c_1 < \dots < c_m \le -1$ according to the positions of the $1$-columns within the $n$ leftmost columns of the alternating sign trapezoid. In the case of vertical symmetry, they are equally distributed on both sides.

Assume that $l \ge 3$. Since the entries of each of the $n$ rows sum up to $1$, there are exactly $n$ columns with column sum $1$. There is no zero in the central column due to symmetry and since the nonzero entries alternate in sign, the central column is $(1,-1,1,\dots,(-1)^{n-1})^\top$. The sum of these entries must be $0$. Hence, $n$ has to be even and $m=n/2$.

If $l=1$, we have to consider the parity of $n$. If $n$ is even, the vertical axis of symmetry is the central column $(1,-1,1,\dots,1,0)^\top$. Thus, there are exactly $n-1$ rows with row sum~$1$ and consequently as many $1$-columns. One of those is the middle column, the remaining $n-2$ are equally distributed on both sides. If $n$ is odd, the axis of symmetry is either $(1,-1,1,\dots,1,-1,1)^\top$ or $(1,-1,1,\dots,1,-1,0)^\top$. In this case, that is, $l = 1$ and $n$ odd, the two subsets of vertically symmetric $(n,1)$-alternating sign trapezoids with bottom entry prescribed to be $0$ and $1$, respectively, are each in bijection with the set of vertically symmetric $(n-1,3)$-alternating sign
trapezoids, by deletion of the bottom entry.

The vertically symmetric $(6,9)$-alternating sign trapezoid given in Figure~\ref{fig:VSASM} has two $10$-columns, four $11$-columns and $1$-column vector~$\mathbf{c}=(-3,-2,-1)$. The central column is given by $(1,-1,1,-1,1,-1)^\top$.

\begin{definition}
We define two weights on vertically symmetric $(n,l)$-alternating sign trapezoids: the \emph{$Q$-weight} is $Q$ raised to the number of $-1$s in the $n-1+(l-1)/2$ leftmost columns; the \emph{$P$-weight} is $P$ raised to the number of $10$-columns within the $n-1$ leftmost columns.
\end{definition}

In order to enumerate vertically symmetric alternating sign trapezoids, we transform them into truncated halved monotone triangles. 

\begin{definition}
For a given integer $n \geq 1$, a \emph{halved monotone triangle of order~$n$} is a triangular array of integers with $n$ rows of one of the following shapes depending on the parity of $n$:
	
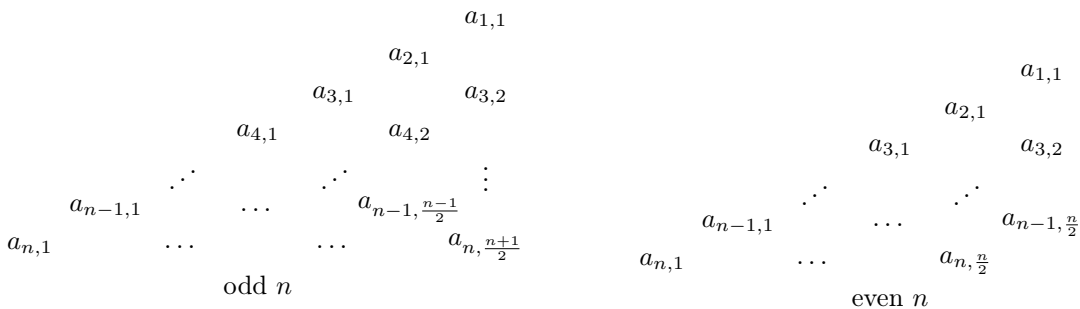
	\begin{center}
	\begin{minipage}{.95\textwidth}
		\begin{minipage}{.5\textwidth}
			\begin{tikzpicture}[yscale=.5]
			\hmt{{$a_{n,1}$,$\dots$,$\dots$,$a_{n,\frac{n+1}{2}}$},
				{$a_{n-1,1}$,$\dots$,$a_{n-1,\frac{n-1}{2}}$},
				{$\iddots$,$\iddots$,$\vdots$},
				{$a_{4,1}$,$a_{4,2}$},
				{$a_{3,1}$,$a_{3,2}$},
				{$a_{2,1}$},
				{$a_{1,1}$}}
			\node at (3,-1) {\text{odd $n$}};
			\end{tikzpicture}
		\end{minipage}\hfill
		\begin{minipage}{.45\textwidth}
			\begin{tikzpicture}[yscale=.5]
			\hmt{{$a_{n,1}$,$\dots$,$a_{n,\frac{n}{2}}$},
				{$a_{n-1,1}$,$\dots$,$a_{n-1,\frac{n}{2}}$},
				{$\iddots$,$\iddots$},
				{$a_{3,1}$,$a_{3,2}$},
				{$a_{2,1}$},
				{$a_{1,1}$}}
			\node at (3,-1) {\text{even $n$}};
			\node at (6,7) {};
			\end{tikzpicture}
		\end{minipage}
	\end{minipage}
	\end{center}
	
The entries
	\begin{compactitem}
		\item 	strictly increase along rows and
		\item 	weakly increase along $\nearrow$-diagonals and $\searrow$-diagonals. 
	\end{compactitem}
\end{definition}

Halved monotone triangles are a generalisation of \emph{vertically symmetric alternating sign matrices}. An \emph{alternating sign matrix} is a square matrix with entries $-1$, $0$, or $+1$ such that the sum of entries in each row and column is $1$, and the nonzero entries alternate in sign along each row and column. If $n$ is odd, there is a bijective correspondence between vertically symmetric alternating sign matrices of size $n \times n$ and halved monotone triangles of order $n-1$ with bottom row $(1,2,\dots,(n-1)/2)$ and no entry larger than $(n-1)/2$. Consider a vertically symmetric alternating sign matrix of size $n \times n$, that is, an $n \times n$ alternating sign matrix that stays invariant under reflection along a vertical axis. By a similar reasoning to that in the case of vertically symmetric alternating sign trapezoids, the vertical axis can only be the middle column $(1,-1,1,\dots,1,-1,1)^\top$. Hence, $n$ has to be odd.

\begin{figure}[htb]
	\centering
	\begin{equation*}
	\begin{pmatrix}
	0 & 0 & 1 & 0 & 0 \\
	1 & 0 & -1 & 0 & 1 \\
	0 & 0 & 1 & 0 & 0 \\
	0 & 1 & -1 & 1 & 0 \\
	0 & 0 & 1 & 0 & 0 
	\end{pmatrix}
	\quad\longleftrightarrow\quad
	\begin{pmatrix}
		0 & 0 & {\color{gray}1} & {\color{gray}0} & {\color{gray}0} \\
		1 & 0 & {\color{gray}0} & {\color{gray}0} & {\color{gray}1} \\
		1 & 0 & {\color{gray}1} & {\color{gray}0} & {\color{gray}1} \\
		1 & 1 & {\color{gray}0} & {\color{gray}1} & {\color{gray}1} \\
		1 & 1 & {\color{gray}1} & {\color{gray}1} & {\color{gray}1} 
	\end{pmatrix}
	\quad\longleftrightarrow\quad
	\begin{array}{cccc}
	&&&1\\
	&&1&\\
	&1&&2\\
	1&&2&
	\end{array}
	\end{equation*}
	\caption{Vertically symmetric alternating sign matrix and the corresponding halved monotone triangle}
	\label{figure:VSASMHMT}
\end{figure}

To each entry of the vertically symmetric alternating sign matrix, we add the entries in the same column above. Thus, we obtain a matrix of $0$s and $1$s. Due to symmetry, we ignore the $(n+1)/2$ rightmost columns, and we record row by row the positions of the $1$s in the remaining columns in the shape of a halved monotone triangle. The resulting halved monotone triangle is of order $n-1$, has bottom row $(1,2,\dots,(n-1)/2)$, and no entry is larger than $(n-1)/2$. An illustrative example is given in Figure~\ref{figure:VSASMHMT}. By considering halved monotone triangles of even and odd order with more general bottom rows, we generalise the notion of vertically symmetric alternating sign matrices.

\begin{definition}
For given integers $m \geq 0$ and $n \geq 1$ with $\lceil n/2 \rceil \geq m$ as well as a weakly decreasing sequence $\mathbf{s}=\left(s_1,s_2,\dots,s_m\right)$ of nonnegative integers, we define a \emph{halved $\mathbf{s}$-tree} as an array of integers which arises from a halved monotone triangle of order~$n$ by truncating the diagonals: for each $1 \leq i \leq m$ we delete the $s_i$ bottom entries of the $i^{\text{th}}$ $\nearrow$-diagonal counted from the left.

We say that a halved $\mathbf{s}$-tree has bottom row~$\mathbf{k}=\left(k_1, \dots, k_{\lceil n/2 \rceil}\right)$ if for all $1 \leq i \leq \lceil n/2 \rceil$ the bottom entry in the $i^{\text{th}}$ $\nearrow$-diagonal counted from left to right is $k_i$.
\end{definition}

Figure~\ref{fig:ShapeTree} illustrates the shape of a halved $(7,3,1)$-tree with $9$ rows. Its bottom row is $(a_{2,1},a_{6,2},a_{8,3},\allowbreak a_{9,4},\allowbreak a_{9,5})$.

\begin{figure}[htb]
	\centering
		\begin{tikzpicture}[yscale=.5,xscale=.8]
		\begin{scope}
		\clip(4,-1) rectangle (9,9);
		\hmt{{,,,$a_{9,4}$,$a_{9,5}$},
			{,,$a_{8,3}$,$a_{8,4}$},
			{,,$a_{7,3}$,$a_{7,4}$},
			{,$a_{6,2}$,$a_{6,3}$},
			{,$a_{5,2}$,$a_{5,3}$},
			{,$a_{4,2}$},
			{,$a_{3,2}$},
			{$a_{2,1}$},
			{$a_{1,1}$}}
		\end{scope}
		\end{tikzpicture}
\caption{Shape of a halved $(7,3,1)$-tree with $9$ rows}
\label{fig:ShapeTree}
\end{figure}

Since trailing zeros in the sequence $\mathbf{s}$ do not affect the shape of the halved $\mathbf{s}$-tree, we set $s_{m+1},\dots,s_{\lceil n/2 \rceil}$ to be $0$ in formulae such as those in Theorems~\ref{thm:QHTREEenumeration}, \ref{thm:QHTREEConstantTerm}, and \ref{thm:HMTPQEnumeration} if $m < \lceil n/2 \rceil  $. Notice that halved trees are a generalisation of halved monotone triangles since the latter can be seen as halved $(0,\dots,0)$-trees. 

\begin{definition}
We call an entry $a_{i,j}$ of a halved tree \emph{special} if the entry $a_{i+1,j}$ exists and
\[
\begin{cases}
a_{i+1,j} < a_{i,j} < a_{i+1,j+1} & \text{if $a_{i+1,j+1}$ exists,}\\
a_{i+1,j} < a_{i,j} & \text{otherwise.}
\end{cases}
\]

As in the case of vertically symmetric alternating sign trapezoids, we define two weights on halved trees: The \emph{$Q$-weight} of a halved tree is defined as $Q$ raised to the number of special entries. It essentially counts the entries that lie strictly between the neighbouring entries in the row below. The \emph{$P$-weight} of a halved tree is $P$ raised to the number of $\nearrow$-diagonals such that the two bottommost entries are equal. See Figure~\ref{figure:vsastrapezoid} for an example.
\end{definition}	
	
\begin{figure}[htb]
	\centering
		\begin{equation*}
		\begin{array}{ccccccccccccccccccc}
		0&0&0&0&0&0&0&0&0&1&0&0&0&0&0&0&0&0&0\\
		&0&0&0&0&0&0&0&1&-1&1&0&0&0&0&0&0&0&\\
		&&0&1&0&0&0&0&-1&1&-1&0&0&0&0&1&0&&\\
		&&&0&0&0&1&0&0&-1&0&0&1&0&0&0&&&\\
		&&&&1&0&-1&0&0&1&0&0&-1&0&1&&&&\\
		&&&&&1&0&0&0&-1&0&0&0&1&&&&&\\
		\end{array}
		\quad\longleftrightarrow\quad
		\begin{array}{ccc}
		&&\\
		&&2\\
		&-3&\\
		-3&&0\\
		&-2&\\
		&&-1
		\end{array}
		\end{equation*}
		\caption{Vertically symmetric $(6,9)$-alternating sign trapezoid with $1$-column vector $\mathbf{c}~=~(-3,-2,-1)$ and weight~$P Q^2$ and its corresponding halved $(2,1)$-tree with bottom row~$(-3,-2,-1)$ and weight~$P Q^2$}
		\label{figure:vsastrapezoid}
\end{figure}

There is a correspondence between vertically symmetric alternating sign trapezoids and certain halved trees. This is a variant of a bijection presented by Fischer \cite{Fis19Trap}. To begin with, we assume that $n$ is even. As stated before, $l$ has to be odd; otherwise, an $(n,l)$-alternating sign trapezoid cannot be vertically symmetric. We demonstrate the modified construction with the example from Figure~\ref{figure:vsastrapezoid}. Given a vertically symmetric ($n$,$l$)-alternating sign trapezoid with $1$-column vector~$\mathbf{c}=(c_1,\dots,c_{n/2})$, we delete all the entries to the right of the vertical symmetry axis and add $0$s on the left to complete the array to an $n \times \left(n + (l-1)/2 \right)$-rectangular shape. We number the columns from $-n$ to $(l-3)/2$ from left to right. Then we replace every entry by the partial column sum of all the entries that lie above it in the same column including the entry itself. This yields an array of integers consisting of $0$s and $1$s. In our example, we get the following array:
\begin{equation*}
	\begin{array}[t]{cccccccccc}
		0&0&0&0&0&0&0&0&0&1\\
		{\color{gray}0}&0&0&0&0&0&0&0&1&0\\
		{\color{gray}0}&{\color{gray}0}&0&1&0&0&0&0&0&1\\
		{\color{gray}0}&{\color{gray}0}&{\color{gray}0}&1&0&0&1&0&0&0\\
		{\color{gray}0}&{\color{gray}0}&{\color{gray}0}&{\color{gray}1}&1&0&0&0&0&1\\
		{\color{gray}0}&{\color{gray}0}&{\color{gray}0}&{\color{gray}1}&{\color{gray}1}&1&0&0&0&0
	\end{array}
\end{equation*}

We record the positions of all $1$s. We proceed row by row beginning at the top and record the column of each nonzero entry. The first row of an alternating sign trapezoid consists of exactly one $1$; in the case of vertical symmetry at position $(l-3)/2$. Copy down this position in the top row of the shape of a halved tree with $n$ rows. Record the positions of the $1$s row by row. Thus, our example turns into the following halved monotone triangle:
\begin{equation*}
	\begin{array}[t]{cccccc}
		&&&&&3\\
		&&&&2&\\
		&&&-3&&3\\
		&&-3&&0&\\
		&{\color{gray}-3}&&-2&&3\\
		{\color{gray}-3}&&{\color{gray}-2}&&-1&
	\end{array}
\end{equation*}

Note that an entry $-1$ in the first $n+(l-3)/2$ columns of the alternating sign trapezoid corresponds to exactly one entry in the halved tree that lies strictly between the neighbouring entries in the row below. The $-1$s in the middle column of the alternating sign trapezoid correspond to the entries $(l-3)/2$ in the right column of the halved tree; they are strictly larger than the left neighbouring entry in the row below. We delete the rightmost column and the entries that originate from the additionally added $0$s at the beginning:
\begin{equation*}
	\begin{array}[t]{ccc}
		&&2\\
		&-3&\\
		-3&&0\\
		&-2&\\
		&&-1
	\end{array}
\end{equation*}

In total, we obtain a halved $(-c_1-1,\dots,-c_{n/2}-1)$-tree with $n-1$ rows and bottom row~$(c_1,\dots,c_{n/2})$ without entries larger than $(l-5)/2$ whose $Q$-weight coincide with the $Q$-weight of the corresponding alternating sign trapezoid.

The following observation is crucial: The two bottommost entries of the $i^{\text{th}}$ $\nearrow$-diagonal of the halved tree -- counted from the left -- are identically $c_i$ if and only if the column $c_i < 0$ of the corresponding alternating sign trapezoid is a $10$-column. Consequently, the $P$-weights of the halved tree and the alternating sign trapezoid coincide, too, and we have the following result:

\begin{proposition}
	\label{pro:Bijection}
	Let $n$ be even. There is a $P$- and $Q$-weight preserving bijective correspondence between vertically symmetric $(n,l)$-alternating sign trapezoids with $1$-column vector~$\mathbf{c}=(c_1,\dots,c_{n/2})$ and halved $(-c_1-1,\dots,-c_{n/2}-1)$-trees with $n-1$ rows and bottom row~$(c_1,\dots,c_{n/2})$ without entries larger than $(l-5)/2$.
\end{proposition}

If $n$ is odd and, thus, $l=1$, the resulting halved trees have different properties than in the demonstrated case of even~$n$. In order to avoid a distinction of cases, we can use the previously made observation that deleting the bottom entry of a vertically symmetric $(n,1)$-alternating sign trapezoid, for $n$ odd, gives a vertically symmetric $(n-1,3)$-alternating sign trapezoid. Hence, the enumeration of vertically symmetric alternating sign trapezoids with an odd number
of rows can be reduced to the enumeration of those with an even number of rows.

\section{Weighted Enumeration of Halved Monotone Triangles and Trees}
\label{sec:EnumHMT}
Halved monotone triangles and trees can be enumerated by so-called operator formulae. This method of enumeration has been initiated by Fischer \cite{Fis06} and developed in a series of follow-up papers \cite{Fis09,Fis10,Fis11,Fis16,Fis18}. We define the following operators: the \emph{identity operator} $\id_x \left[f(x)\right] \coloneqq f(x)$, the \emph{shift operator} $\E_x \left[f(x)\right] \coloneqq f(x+1)$, the \emph{forward difference} $\fd_x \coloneqq \E_x-\id_x$, and the generalised \emph{$Q$-forward difference} $\Qfd_x \coloneqq \fd_x (\id_x - (1-Q)\E_x)^{-1}$. Given a variable~$x$ and an integer~$a$, we use the notation $\E_a \left[f(a)\right] \coloneqq \left. \E_x \left[f(x)\right]\right|_{x=a}$ and similarly for other operator expressions. Note that for a function of several variables, any of these operators which is associated with a variable $x$ commutes with any of these operators which is associated with a variable $y$. For the sake of convenience, we omit the variable the identity operator $\id_x$ refers to and write $\id$ instead throughout the paper. Also note that the $Q$-forward difference $\Qfd_x$ reduces to the forward difference $\fd_x$ if we set $Q=1$. We define more operators in Section~\ref{sec:Proofs}. Table~\ref{tab:Op} summarises all the operators we use in this paper.

The following theorem is our main result on the refined enumeration of halved monotone triangles. It provides an operator formula for the sum of the $Q$-weights of all halved monotone triangles of order~$n$ with prescribed bottom row~$\mathbf{k}$ and no entry larger than $b$, which we denote by $\QHMT_n\left(b;\mathbf{k}\right)$. Theorem~\ref{thm:QHMTenumeration} generalises the straight enumeration \cite[Theorem 1]{Fis09}, which can be recovered by setting $Q=1$.
\begin{theorem}\label{thm:QHMTenumeration}
	The $Q$-generating function $\QHMT_n\left(b;\mathbf{k}\right)$ of halved monotone triangles of order~$n$ with prescribed bottom row~$\mathbf{k}=(k_1,\dots,k_{\lceil n/2 \rceil})$ and no entry larger than $b$ is given by
	\begin{multline*}
		\prod_{1\leq s<t\leq \frac{n+1}{2}} \left(\E_{k_s}+\E_{k_t}^{-1}-(2-Q)\E_{k_s}\E_{k_t}^{-1}\right) \left(\E_{k_s}+\E_{k_t}-(2-Q)\E_{k_s}\E_{k_t}\right) \\
		\left[\prod_{1\leq i<j\leq \frac{n+1}{2}} \frac{(k_j-k_i+j-i)(2b+n+2-k_i-k_j-i-j)}{(j-i)(i+j-1)}\right]
	\end{multline*}
	if $n\in\mathbb{N}$ is odd and by
	\begin{multline*}
		\prod_{r=1}^{\frac{n}{2}} \left((Q-1) \E_{k_r}+\id\right) \prod_{1\leq s<t\leq \frac{n}{2}} \left(\E_{k_s}+\E_{k_t}^{-1}-(2-Q)\E_{k_s}\E_{k_t}^{-1}\right) \left(\E_{k_s}+\E_{k_t}-(2-Q)\E_{k_s}\E_{k_t}\right) \\
		\left[\prod_{1\leq i<j\leq \frac{n}{2}}\frac{(k_j-k_i+j-i)(2b+n+2-k_i-k_j-i-j)}{(j-i)(i+j)} \prod_{i=1}^{\frac{n}{2}}\frac{b+\frac{n}{2}+1-k_i-i}{i}\right]
	\end{multline*}
	if $n\in\mathbb{N}$ is even.
\end{theorem}

Halved trees are our main tool for enumerating vertically symmetric alternating sign trapezoids. They can be enumerated by an operator formula as shown in the following theorem. We obtain the generating function of halved trees by applying generalised difference operators to the generating functions in Theorem~\ref{thm:QHMTenumeration}.

\begin{theorem}\label{thm:QHTREEenumeration}
	The $Q$-generating function of halved $\mathbf{s}$-trees with prescribed bottom row~$\mathbf{k}$ and no entry greater than $b$ is given by
	\begin{equation}\label{eq:QHTREEenumeration}
		\prod_{r=1}^{\lceil\frac{n}{2}\rceil} \left( -\Qfd_{k_r} \right)^{s_r} \QHMT_n(b;\mathbf{k}).
	\end{equation}
\end{theorem}

Note that $\Qfd_x = Q^{-1} \sum_{i=0}^{\infty} (Q^{-1}-1)^i \fd_x^{i+1}$ applied to a polynomial~$f$ in $x$ becomes a finite sum because $\fd_x^{i+1} \left[f(x)\right]$ eventually vanishes for large enough~$i$. Hence, the expression in Theorem~\ref{thm:QHTREEenumeration} is well defined.

In the next theorem, we reformulate~\eqref{eq:QHTREEenumeration} as constant term identities.

\begin{theorem}\label{thm:QHTREEConstantTerm}
	The $Q$-generating function for halved $\mathbf{s}$-trees with prescribed bottom row~$\mathbf{k}$ and no entry larger than $b$ is the constant term in $x_1$, \dots, $x_{(n+1)/2}$ of 	
	\begin{multline*}
		\prod_{r=1}^{\frac{n+1}{2}} x_r^{1-n} (1+x_r)^{k_r-b-\frac{n+1}{2}} \left(-\frac{x_r}{Q-(1-Q)x_r}\right)^{s_r} \prod_{1\leq s<t\leq \frac{n+1}{2}} \left( x_t - x_s\right) \left( x_s + x_t + x_s x_t \right) \\
		\times \left( Q + (Q-1)x_s + x_t + x_s x_t \right) \left( Q + (Q-1)x_s + (Q-1)x_t + (Q-2)x_s x_t \right)
	\end{multline*}
	if $n$ is odd and the constant term in $x_1$, \dots, $x_{n/2}$ of 	
	\begin{multline*}
		(-1)^{\frac{n}{2}} \prod_{r=1}^{\frac{n}{2}} x_r^{1-n}  \left( Q - (1-Q)x_r \right) (1+x_r)^{k_r-b-\frac{n}{2}} \left(-\frac{x_r}{Q-(1-Q)x_r}\right)^{s_r} \prod_{1\leq s<t\leq \frac{n}{2}} \left( x_t - x_s\right) \\
		\times \left( x_s + x_t + x_s x_t \right) \left( Q + (Q-1)x_s + x_t + x_s x_t \right) \left( Q + (Q-1)x_s + (Q-1)x_t + (Q-2)x_s x_t \right)
	\end{multline*}
	if $n$ is even.	
\end{theorem}

If we set $Q=1$ and $\mathbf{s}=(0,\dots,0)$, we obtain constant term identities for the number of halved monotone triangles. Other constant term identities for the number of halved monotone triangles have already been established \cite[(6.15) \& (6.16)]{Fis09} but they are only valid if $k_i \le b+1-(n+1)/2$ or $k_i \le b+2-n/2$ for the case that $n$ is odd or even, respectively. In contrast, there are no such constraints in Theorem~\ref{thm:QHTREEConstantTerm}.

The following theorem establishes a generating function of halved trees where we impose certain conditions on the bottommost entries of the diagonals. This becomes useful for the $P$-generating function of vertically symmetric alternating sign trapezoids in Theorem~\ref{thm:VSASTPQCEnumeration}.
\begin{theorem}\label{thm:HMTPQEnumeration}
	Let $L_{=} \subseteq \{1,\dots,\lceil n/2 \rceil \}$. The $Q$-generating function of halved $\mathbf{s}$-trees with prescribed bottom row~$\mathbf{k}$ and no entry greater than $b$ such that the two bottommost entries in the $i^\text{th}$ diagonal are equal if $i \in L_{=}$ and different if $i \notin L_{=}$ is given by
	\begin{equation*}
	\prod_{i \in L_{=}} \left( -\Qfd_{k_i} \right) \prod_{\substack{1 \leq i \leq {\lceil \frac{n}{2} \rceil},\\ i \notin L_{=}}} \left( \id + \Qfd_{k_i} \right) \prod_{r=1}^{\lceil\frac{n}{2}\rceil} \left( -\Qfd_{k_r} \right)^{s_r} \QHMT_n(b;\mathbf{k}).
	\end{equation*}
\end{theorem}

\section{Application to Vertically Symmetric Alternating Sign Trapezoids}
\label{sec:EnumAST}

We apply our previous results to the case of vertically symmetric $(n,l)$-alternating sign trapezoids in order to derive a generating function. To this end, we adapt the ideas of \cite{Fis19} to our setting.

Our first result about vertically symmetric alternating sign trapezoids is based on Theorem~\ref{thm:HMTPQEnumeration} and simply follows from Proposition~\ref{pro:Bijection} about the bijection between vertically symmetric alternating sign trapezoids and halved trees. We start with the case of even~$n$.

\begin{theorem}\label{thm:VSASTQCEnumeration}
	Let $n$ be even and $l$ odd and let $C_{10} \subseteq \left\{ c_1,\dots,c_{n/2} \right\}$. The $Q$-generating function of vertically symmetric $(n,l)$-alternating sign trapezoids with $1$-columns in positions $\mathbf{c}=\left(c_1,\dots,c_{n/2}\right)$ such that $c_i$ is a $10$-column if and only if $c_i \in C_{10}$ is given by
	\begin{equation}\label{eq:VSASTQCEnumeration}
	\prod_{c_i \in C_{10}} \left( -\Qfd_{c_i} \right) \prod_{\substack{1 \leq i \leq {\frac{n}{2}},\\ c_i \notin C_{10}}} \left( \id + \Qfd_{c_i} \right) \prod_{r=1}^{\frac{n}{2}} \left( -\Qfd_{c_r} \right)^{-c_r-1} \QHMT_{n-1} \left( \frac{l-5}{2};\mathbf{c} \right).
	\end{equation}
\end{theorem}
In the following theorem, we derive the $PQ$-generating function of vertically symmetric alternating sign trapezoids with a given distribution of $1$-columns. 
\begin{theorem}\label{thm:VSASTPQCEnumeration}
	Let $n$ be even  and $l$ odd. The $PQ$-generating function of vertically symmetric $(n,l)$-alternating sign trapezoids with $1$-columns in positions $\mathbf{c}=\left(c_1,\dots,c_{n/2}\right)$ is given by
	\begin{equation*}
		\prod_{r=1}^{\frac{n}{2}} \left( \id - (P-1) \Qfd_{c_r} \right) \left( -\Qfd_{c_r} \right)^{-c_r-1} \QHMT_{n-1} \left( \frac{l-5}{2};\mathbf{c} \right).
	\end{equation*}
	This is equal to the constant term in $x_1$, \dots, $x_{n/2}$ of
	\begin{multline*}
		\prod_{r=1}^{\frac{n}{2}} x_r^{2-n} \left(1+x_r\right)^{c_r-\frac{l-5}{2}-\frac{n}{2}} \left( \frac{Q-(P-Q)x_r}{Q-(1-Q)x_r} \right) \left(- \frac{x_r}{Q - (1-Q)x_r}\right)^{-c_r-1}\\
		\times \prod_{1\leq s<t\leq \frac{n}{2}} \left( \left( x_t - x_s\right) \left( x_s + x_t + x_s x_t \right) \right.\\
		\left. \times \left( Q + (Q-1)x_s + x_t + x_s x_t \right) \left( Q + (Q-1)x_s + (Q-1)x_t + (Q-2)x_s x_t \right) \right).
	\end{multline*}
\end{theorem}
In the next theorem, we provide the generating function of all vertically symmetric $(n,l)$-alternating sign trapezoids without prescribing the positions of the $1$-columns.
\begin{theorem}\label{thm:VSASTPQEnumeration}
	Let $n$ be even and $l$ odd. The $PQ$-generating function of vertically symmetric $(n,l)$-alternating sign trapezoids is given by the constant term in $x_1$, \dots, $x_{n/2}$ of
	\begin{multline*}
		\frac{1}{\left(\frac{n}{2}\right)!} \prod_{r=1}^{\frac{n}{2}} \frac{ Q+(Q-P) x_r}{x_r^{n-2}  \left(1+x_r\right)^{\frac{l-5}{2}+\frac{n}{2}} \left(Q(1+x_r)^2 - x_r^2\right)} \\
		\times \prod_{1\leq s<t\leq \frac{n}{2}} \frac{\left( x_t - x_s\right)^2 \left( x_s + x_t + x_s x_t \right) \left( Q + (Q-1)x_s + (Q-1)x_t + (Q-2)x_s x_t \right) \left(Q-x_s x_t\right)}{Q(1+x_s)(1+x_t)-x_s x_t}.
	\end{multline*}
\end{theorem}
If $n$ is odd, the bottom row of a vertically symmetric $(n,1)$-alternating sign trapezoid is either $1$ or $0$. We can delete this row and obtain a vertically symmetric $(n-1,3)$-alternating sign trapezoid as aforementioned. This observation implies the following theorem. 
\begin{corollary}\label{thm:VSASTPQEnumerationOdd}
	Let $n$ be odd. The $PQ$-generating function of vertically symmetric $(n,1)$-alternating sign trapezoids is given by the constant term in $x_1$, \dots, $x_{(n-1)/2}$ of
	\begin{multline*}
		\frac{2}{\left(\frac{n-1}{2}\right)!} \prod_{r=1}^{\frac{n-1}{2}} \frac{ Q+(Q-P) x_r}{x_r^{n-3}  \left(1+x_r\right)^{\frac{n-3}{2}} \left(Q(1+x_r)^2 - x_r^2\right)} \\
		\times \prod_{1\leq s<t\leq \frac{n-1}{2}} \frac{\left( x_t - x_s\right)^2 \left( x_s + x_t + x_s x_t \right) \left( Q + (Q-1)x_s + (Q-1)x_t + (Q-2)x_s x_t \right) \left(Q-x_s x_t\right)}{Q(1+x_s)(1+x_t)-x_s x_t}.
	\end{multline*}
\end{corollary}
Vertically symmetric alternating sign triangles of order~$n$ -- the order~$n$ has to be odd -- can be transformed into vertically symmetric $(n-1,3)$-alternating sign trapezoids by cutting off the entry $1$ in the bottom row. If we define the $P$-weight and $Q$-weight in a similar way for alternating sign triangles, we obtain the following generating function:
\begin{corollary}\label{thm:VSASTrianglePQEnumeration}
	The $PQ$-generating function of vertically symmetric alternating sign triangles of order~$n$ is given by the constant term in $x_1$, \dots, $x_{(n-1)/2}$ of
	\begin{multline*}
		\frac{1}{\left(\frac{n-1}{2}\right)!} \prod_{r=1}^{\frac{n-1}{2}} \frac{ Q+(Q-P) x_r}{x_r^{n-3}  \left(1+x_r\right)^{\frac{n-3}{2}} \left(Q(1+x_r)^2 - x_r^2\right)} \\
		\times \prod_{1\leq s<t\leq \frac{n-1}{2}} \frac{\left( x_t - x_s\right)^2 \left( x_s + x_t + x_s x_t \right) \left( Q + (Q-1)x_s + (Q-1)x_t + (Q-2)x_s x_t \right) \left(Q-x_s x_t\right)}{Q(1+x_s)(1+x_t)-x_s x_t}.
	\end{multline*}
\end{corollary}

\section{Proofs}
\label{sec:Proofs}
In this section, we provide the proofs for the statements in Section~\ref{sec:EnumHMT} and \ref{sec:EnumAST}. The involved formulae comprise many different operators. Table~\ref{tab:Op} gives an overview of the basic and $Q$-generalised operators we use. By setting $Q=1$, we obtain the corresponding nongeneralised operators. Note that we do not use the nongeneralised version of $\T_{x,y}$.

In addition, we introduce the following notation for constant term formulae. We denote by
\begin{equation*}
\ct_{x_1,\dots,x_m} \left(f(x_1,\dots,x_m)\right)
\end{equation*}
the constant term of a formal Laurent series~$f$, that means the coefficient of the term $x_1^0 \cdots x_m^0$.
	
	\begin{table}[htb]
		\centering
		\caption{Overview of basic and generalised operators}
		\begin{tabular}{ll}
			\toprule
			identity operator & $\id_x \left[f(x)\right] \coloneqq f(x)$ \\
			shift operator & $\E_x \left[f(x)\right] \coloneqq f(x+1)$\\
			forward difference & $\fd_x \coloneqq \E_x-\id_x$\\
			backward difference & $\bd_x \coloneqq \id_x-\E_x^{-1}$\\
			strict operator& $\Strict_{x,y} \coloneqq \E_{x}^{-1}+\E_{y}-\E_{x}^{-1}\E_{y}$\\
			\midrule
			$Q$-identity operator & $\Qid_x \coloneqq (Q-1) \E_x + \id_x$\\
			$Q$-shift operator & $\QE_x \coloneqq (Q-1) \id_x + \E_x$\\
			$Q$-forward difference & $\Qfd_x \coloneqq \fd_x (\id_x - (1-Q)\E_x)^{-1}$\\
			$Q$-strict operator & $\QStrict_{x,y} \coloneqq \E_{x}^{-1}+\E_{y}-(2-Q)\E_{x}^{-1}\E_{y}$\\
			modified $Q$-strict operator
			& $\T_{x,y} \coloneqq \E_x + \E_y - (2-Q)\E_x \E_y$\\
			\bottomrule
		\end{tabular}
		\label{tab:Op}
	\end{table}

\subsection{Proof of Theorem~\ref{thm:QHMTenumeration}}

Theorem~\ref{thm:QHMTenumeration} is the basic result on the enumeration of halved monotone triangles. Its proof is split up into several steps. First, Lemma~\ref{lem:DetFormulae} enables us to rewrite the operands appearing in Theorem~\ref{thm:QHMTenumeration} using determinants. Secondly, we observe how to recursively build up halved monotone triangles which leads to the definition of certain summation operators. Thirdly, in Lemma~\ref{lem:SumOpNormal} and \ref{lem:SumOpAlt}, we see how to apply the summation operators to certain polynomials. Finally, in Lemma~\ref{lem:AppSumOdd} and \ref{lem:AppSumEven}, we apply these results to the operands in Theorem~\ref{thm:QHMTenumeration}.

\begin{lemma}
	\label{lem:DetFormulae}
	The following determinant evaluations hold true:
	\begin{align}
	\det_{1 \leq i,j \leq n} \left( \binom{k_i+j-1}{2j-1} \right)
	&= \prod_{1 \leq i < j \leq n}\frac{(k_j-k_i)(k_j+k_i)}{(j-i)(j+i)} \prod_{i=1}^n\frac{k_i}{i}, \label{eq:DetFormulaeEven}\\
	\det_{1 \leq i,j \leq n} \left( \binom{k_i+j-\frac{3}{2}}{2j-2} \right) &= \prod_{1 \leq i < j \leq n} \frac{(k_j-k_i)(k_j+k_i)}{(j-i)(j+i-1)}. \label{eq:DetFormulaeOdd}
	\end{align}
\end{lemma}

Lemma~\ref{lem:DetFormulae} is proved as Lemma~13 in \cite{Fis09}. Moreover, Lemma~\ref{lem:DetFormulae} follows from Theorem~27 in \cite{Kra01}. In particular, replacing $j$ by $n + 1 - j$ on the left-hand side of \cite[Theorem~27]{Kra01} with $q = 1$ and then setting
$L_i = k_i - n - 1$ and $A = 2n + 1$ yields Equation~\eqref{eq:DetFormulaeEven}, while setting $L_i = k_i - n - 1/2$ and $A = 2n$ yields Equation~\eqref{eq:DetFormulaeOdd}. 

If we replace $k_i$ by $k_i+i-b-n/2-1$ and $n$ by $n/2$ for even $n$ as well as $k_i$ by $k_i+i-b-(n+1)/2-1/2$ and $n$ by $(n+1)/2$ for odd $n$ in the Equations~\eqref{eq:DetFormulaeEven} and \eqref{eq:DetFormulaeOdd}, respectively, we obtain the following two evaluations:

If $n$ is even, then
\begin{multline}\label{eq:HMTDetEven}
	\prod_{1\leq i<j\leq \frac{n}{2}}\frac{(k_j-k_i+j-i)(2b+n+2-k_j-k_i-j-i)}{(j-i)(j+i)} \prod_{i=1}^{\frac{n}{2}}\frac{b+\frac{n}{2}+1-k_i-i}{i}\\
	= (-1)^{\binom{\frac{n}{2}+1}{2}} \det_{1\leq i,j\leq\frac{n}{2}} \left( \binom{k_i+i+j-b-\frac{n}{2}-2}{2j-1} \right)
\end{multline}
and if $n$ is odd, then
\begin{multline}\label{eq:HMTDetOdd}
	\prod_{1\leq i<j\leq \frac{n+1}{2}}\frac{(k_j-k_i+j-i)(2b+n+2-k_j-k_i-j-i)}{(j-i)(j+i-1)}\\
	= (-1)^{\binom{\frac{n+1}{2}}{2}} \det_{1\leq i,j\leq\frac{n+1}{2}} \left( \binom{k_i+i+j-b-\frac{n+1}{2}-2}{2j-2} \right).
\end{multline}

We take advantage of the recursive structure of halved monotone triangles: If we cut off the bottom row of a halved monotone triangle of order~$n$, we obtain a halved monotone triangle of order~$n-1$.

\begin{figure}[htb]
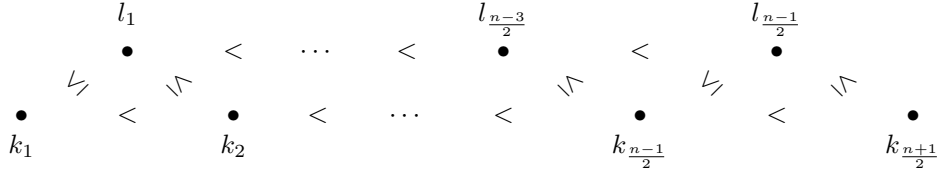

	\centering
	\begin{equation*}
	\begin{array}[t]{ccccccccccccccccc}
	&&l_1&&&&&&&&l_{\frac{n-3}{2}}&&&&l_{\frac{n-1}{2}}&&\\
	&&\bullet&&<&&\cdots&&<&&\bullet&&<&&\bullet&&\\
	&\lle&&\rle&&&&&&&&\rle&&\lle&&\rle&\\
	\bullet&&<&&\bullet&&<&&\cdots&&<&&\bullet&&<&&\bullet\\
	k_1&&&&k_2&&&&&&&&k_{\frac{n-1}{2}}&&&&k_{\frac{n+1}{2}}
	\end{array}
	\end{equation*}
	\caption{Bottom and penultimate row of a halved monotone triangle of odd order}
	\label{fig:BottomRowsOdd}
\end{figure}

Suppose $n$ is odd. Figure~\ref{fig:BottomRowsOdd} shows the bottom and penultimate row of a halved monotone triangle of order $n$. In order to enumerate all halved monotone triangles of order~$n$ with bottom row $(k_1,\dots,k_{(n+1)/2})$ and no entry larger than $b$, we have to sum over all halved monotone triangles of order~$n-1$ with bottom row $(l_1,\dots,l_{(n-1)/2})$ and no entry larger than $b$ such that $l_1 < \dots < l_{(n-1)/2}$ and $k_1 \le l_1 \le k_2 \le \dots \le k_{(n-1)/2} \le l_{(n-1)/2} \le k_{(n+1)/2}$. Taking the $Q$-weight into account, this observation motivates the definition of the following \emph{summation operator} for arbitrary $n$:
\begin{equation*}
\Qsum{n-1}{n} f(l_1, l_2, \dots,l_{n-1}) \coloneqq \sum_{\substack{ l_1 < l_2 < \dots < l_{n-1}, \\ k_1 \leq l_1 \leq k_2 \leq \dots \leq k_{n-1} \leq l_{n-1} \leq k_n}} \hspace*{-1ex} Q^{[k_{1} < l_{1} < k_{2}]+\dots+[k_{n-1} < l_{n-1} < k_{n}]} f(l_1, \dots, l_{n-1}),
\end{equation*}
where we make use of the \emph{Iverson bracket}: For any logical proposition~$P$, $[P]=1$ if $P$ is satisfied and $[P]=0$ otherwise.

If we set $Q=1$, Fischer \cite{Fis16} showed how to approach the summation problem by means of operators: 

\begin{proposition}\label{pro:strictFischer}
	The following operator formula holds:
\begin{equation}
	\label{eq:strictFischer}
	\sum_{\substack{l_{i-1} < l_{i}, \\ k_{i-1} \leq l_{i-1} \leq k_{i} \leq l_{i} \leq k_{i+1}}} f(l_{i-1},l_{i})\\
	=
	\left.\left( 
	\Strict_{k_{i}^{(1)},k_{i}^{(2)}}
	\left[ \sum_{l_{i-1}=k_{i-1}^{(2)}}^{k_{i}^{(1)}} \sum_{l_{i}=k_{i}^{(2)}}^{k_{i+1}^{(1)}} f(l_{i-1},l_{i}) \right] \right)\right|_{k_{i}^{(1)}=k_{i}^{(2)}=k_{i}}.
\end{equation}
\end{proposition}

If we omitted the operator on the right-hand side of Equation~\eqref{eq:strictFischer}, we would sum over all $l_{i-1} \le l_{i}$ such that $k_{i-1} \leq l_{i-1} \leq k_{i} \leq l_{i} \leq k_{i+1}$; that is, the \emph{strict operator} ensures the strict monotonicity $l_{i-1} < l_{i}$. We want to generalise Proposition~\ref{pro:strictFischer} by incorporating the $Q$-weight. To this end, we introduce the \emph{$Q$-strict operator}. By a straightforward computation, it holds that
\begin{multline}
	\label{eq:QStrict}
Q^{-1} \left.\left( \QStrict_{k_{i}^{(1)},k_{i}^{(2)}} \QStrict_{k_{i+1}^{(1)},k_{i+1}^{(2)}}  \left[\sum_{l_{i}=k_{i}^{(2)}}^{k_{i+1}^{(1)}} f(l_{i}) \right] \right)\right|_{\substack{k_{j}^{(1)}=k_{j}^{(2)}=k_{j} \\ \forall j\in\{i,i+1\}}}\\
= Q^{-1} \left(\sum_{l_{i}=k_{i}}^{k_{i+1}} f(l_{i}) + \left(Q-1\right) \sum_{l_{i}=k_{i}+1}^{k_{i+1}} f(l_{i}) + \left(Q-1\right) \sum_{l_{i}=k_{i}}^{k_{i+1}-1} f(l_{i}) + \left(Q-1\right)^2 \sum_{l_{i}=k_{i}+1}^{k_{i+1}-1} f(l_{i})\right)\\
=\sum_{l_i=k_i}^{k_{i+1}} Q^{\left[k_{i} < l_{i} < k_{i+1}\right]} f(l_{i}).
\end{multline}
By extending Equation~\eqref{eq:QStrict}, we see that $\Qsum{n-1}{n} f(l_1, l_2, \dots,l_{n-1})$ is equivalent to
\begin{equation}\label{def:SumOp}
	Q^{-1} \left.\left(\QStrict_{k_{1}^{(1)},k_{1}^{(2)}} \cdots \QStrict_{k_{n}^{(1)},k_{n}^{(2)}}
	\left[\sum_{l_{1}=k_{1}^{(2)}}^{k_{2}^{(1)}} \cdots \hspace{-1ex} \sum_{l_{n-1}=k_{n-1}^{(2)}}^{k_{n}^{(1)}} f(l_1, \dots, l_{n-1})\right] \right) \right|_{k_{i}^{(1)}=k_{i}^{(2)}=k_{i}}.
\end{equation}

Figure~\ref{fig:BottomRowsEven} shows the bottom row $(k_1,\dots,k_{n/2})$ and the penultimate row $(l_1,\dots,l_{n/2})$ of a halved monotone triangle of even order $n$ and no entry larger than $b$. The entry $l_{n/2}$ is only bounded from above by $b$, not by an entry of the bottom row.

\begin{figure}[htb]
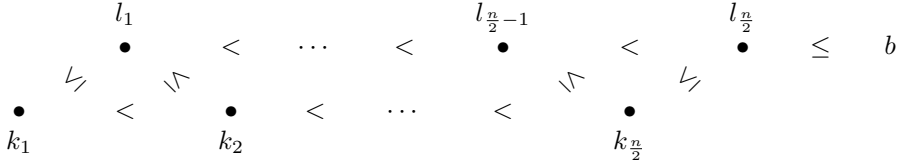

	\centering
	\begin{equation*}
	\begin{array}[t]{ccccccccccccccccccc}
	&&l_1&&&&&&&&l_{\frac{n}{2}-1}&&&&l_{\frac{n}{2}}&&&&\\
	&&\bullet&&<&&\cdots&&<&&\bullet&&<&&\bullet&&\le&&b\\
	&\lle&&\rle&&&&&&&&\rle&&\lle&&&&&\\
	\bullet&&<&&\bullet&&<&&\cdots&&<&&\bullet&&&&&&\\
	k_1&&&&k_2&&&&&&&&k_{\frac{n}{2}}&&&&&&
	\end{array}
	\end{equation*}
	\caption{Bottom and penultimate row of a halved monotone triangle of even order}
	\label{fig:BottomRowsEven}
\end{figure}

This observation leads to the following alternative version of the summation operator:
\begin{equation*}
\sum_{\substack{ l_1 < l_2 < \dots < l_{n-1}, \\ k_1 \leq l_1 \leq k_2 \leq \dots \leq k_{n-1} \leq l_{n-1} \leq k_n}} \hspace*{-1ex} Q^{[k_{1} < l_{1} < k_{2}]+\dots+[k_{n-2} < l_{n-2} < k_{n-1}]+[k_{n-1} < l_{n-1}]} f(l_1, \dots, l_{n-1}),
\end{equation*}
which we define for arbitrary $n$ and denote by $\AQsum{n-1}{n} f(l_1, l_2, \dots,l_{n-1})$. Since
\begin{equation*}
	\left. \left(\QStrict_{k_{i}^{(1)},k_{i}^{(2)}} \left[\sum_{l_{i}=k_{i}^{(2)}}^{k_{i+1}^{(1)}} f(l_{i})\right] \right) \right|_{\substack{k_{j}^{(1)}=k_{j}^{(2)}=k_{j} \\ \forall j\in\{i,i+1\}}}
	= \sum_{l_{i}=k_{i}}^{k_{i+1}} f(l_{i}) + \left(Q-1\right) \sum_{l_{i}=k_{i}+1}^{k_{i+1}} f(l_{i}) 
	=\sum_{l_i=k_i}^{k_{i+1}} Q^{\left[k_{i} < l_{i}\right]} f(l_{i}),
\end{equation*}
we can conclude that
\begin{multline}\label{def:AltSumOp}
	\AQsum{n-1}{n} f(l_1, l_2, \dots,l_{n-1})\\
	= \left.\left(\QStrict_{k_{1}^{(1)},k_{1}^{(2)}} \cdots \QStrict_{k_{n-1}^{(1)},k_{n-1}^{(2)}} \left[\sum_{l_{1}=k_{1}^{(2)}}^{k_{2}^{(1)}} \cdots \hspace{-1ex} \sum_{l_{n-1}=k_{n-1}^{(2)}}^{k_{n}^{(1)}} f(l_1, \dots, l_{n-1})\right] \right) \right|_{k_{i}^{(1)}=k_{i}^{(2)}=k_{i}}.
\end{multline}
We are now in a position to state a recursive formula for the generating functions $\QHMT_n\left(b;\mathbf{k}\right)$:
\begin{align*}
	\QHMT_n\left(b;\left(k_1,\dots,k_{\frac{n+1}{2}}\right)\right) &= \Qsum{\frac{n-1}{2}}{\frac{n+1}{2}} \QHMT_{n-1}\left(b;\left(l_1,\dots,l_{\frac{n-1}{2}}\right)\right) &\text{if $n$ is odd,}\\
	\QHMT_n\left(b;\left(k_1,\dots,k_{\frac{n}{2}}\right)\right) &= \AQbsum{\frac{n}{2}} \QHMT_{n-1}\left(b;\left(l_1,\dots,l_{\frac{n}{2}}\right)\right) &\text{if $n$ is even.}
\end{align*}

In the next two lemmata, we show how to apply the summation operator and its variation to certain kinds of polynomials. Lemma~\ref{lem:SumOpNormal} is a corollary of \cite[Lemma 1]{Fis10}, and Lemma~\ref{lem:SumOpAlt} is a variation thereof.


\begin{lemma}\label{lem:SumOpNormal}
	Let $g(l_1,\dots ,l_{n-1})$ be a polynomial in $\left(l_1,\dots ,l_{n-1}\right)$ such that \newline $\left.\QStrict_{l_i,l_{i+1}} \left[g(l_1,\dots ,l_{n-1})\right] \right|_{l_i=l_{i+1}+1}$ vanishes for every $i\in\{1,\dots ,n-2\}$. Then
	\begin{equation*}
		\Qsum{n-1}{n}\prod_{i=1}^{n-1}\fd_{l_i} \left[g(l_1,\dots,l_{n-1})\right] = \sum_{r=1}^n (-1)^{r-1}\prod_{s=1}^{r-1} \Qid_{k_s}\prod_{t=r+1}^{n} \QE_{k_t} \left[g(k_1,\dots ,\widehat{k_r},\dots ,k_n)\right],
	\end{equation*}
	where $g(k_1,\dots ,\widehat{k_r},\dots ,k_n) \coloneqq g(k_1,\dots ,k_{r-1},k_{r+1},\dots ,k_n)$.
\end{lemma}

\begin{lemma}\label{lem:SumOpAlt}
	Let $g(l_1,\dots ,l_{n-1})$ be a polynomial in $\left(l_1,\dots ,l_{n-1}\right)$ such that \newline $\left.\QStrict_{l_i,l_{i+1}} \left[g(l_1,\dots ,l_{n-1})\right] \right|_{l_i=l_{i+1}+1}$ vanishes for every $i\in\{1,\dots ,n-2\}$. Then
	\begin{multline*}
		\AQsum{n-1}{n} \prod_{i=1}^{n-1}\fd_{l_i} \left[g(l_1,\dots,l_{n-1})\right] = 
		Q \sum_{r=1}^{n-1} (-1)^{r-1}\prod_{s=1}^{r-1} \Qid_{k_s}\prod_{t=r+1}^{n-1} \QE_{k_t} \left[g(k_1,\dots ,\widehat{k_r},\dots,k_{n-1},k_n+1)\right]\\
		+ (-1)^{n-1}\prod_{s=1}^{n-1} \Qid_{k_s} \left[g(k_1,\dots,k_{n-1})\right].
	\end{multline*}
\end{lemma}

\begin{proof}[Proof of Lemma~\ref{lem:SumOpAlt}]
	First, we use Equation~\eqref{def:AltSumOp} to create telescoping sums:	
	\begin{multline*}
		\AQsum{n-1}{n}\prod_{i=1}^{n-1}\fd_{l_i} \left[g(l_1,\dots,l_{n-1})\right] \\
		\shoveleft = \left.\left(\QStrict_{k_{1}^{(1)},k_{1}^{(2)}} \cdots \QStrict_{k_{n-1}^{(1)},k_{n-1}^{(2)}} \sum_{l_{1}=k_{1}^{(2)}}^{k_{2}^{(1)}} \cdots \hspace{-1ex} \sum_{l_{n-1}=k_{n-1}^{(2)}}^{k_{n}^{(1)}} \prod_{i=1}^{n-1}\fd_{l_i} \left[g(l_1,\dots,l_{n-1})\right] \right) \right|_{k_{i}^{(1)}=k_{i}^{(2)}=k_{i}} \\
		\shoveleft = \left.\left(\QStrict_{k_{1}^{(1)},k_{1}^{(2)}} \cdots \QStrict_{k_{n-1}^{(1)},k_{n-1}^{(2)}} \sum_{l_{1}=k_{1}^{(2)}}^{k_{2}^{(1)}} \cdots \hspace{-1ex} \sum_{l_{n-2}=k_{n-2}^{(2)}}^{k_{n-1}^{(1)}} \prod_{i=1}^{n-2}\fd_{l_i} \right.\right.\\
		\left.\left. \vphantom{\sum_{l_{n-2}=k_{n-2}^{(2)}}^{k_{n-1}^{(1)}}}  \left[g(l_1,\dots,l_{n-2},k_{n}^{(1)}+1) - g(l_1,\dots,l_{n-2},k_{n-1}^{(2)})\right] \right) \right|_{k_{i}^{(1)}=k_{i}^{(2)}=k_{i}}\\
		\shoveleft = Q \Qsum{n-2}{n-1}\prod_{i=1}^{n-2}\fd_{l_i} \left[g(l_1,\dots,l_{n-2},k_n+1)\right]\\
		- \left.\left(\QStrict_{k_{1}^{(1)},k_{1}^{(2)}} \cdots \QStrict_{k_{n-1}^{(1)},k_{n-1}^{(2)}} \sum_{l_{1}=k_{1}^{(2)}}^{k_{2}^{(1)}} \cdots \hspace{-1ex} \sum_{l_{n-2}=k_{n-2}^{(2)}}^{k_{n-1}^{(1)}} \prod_{i=1}^{n-2}\fd_{l_i} \left[g(l_1,\dots,l_{n-2},k_{n-1}^{(2)})\right] \right) \right|_{k_{i}^{(1)}=k_{i}^{(2)}=k_{i}}.
	\end{multline*}
	Using Lemma~\ref{lem:SumOpNormal}, it still remains to be shown that
	\begin{multline*}
		- \left.\left(\QStrict_{k_{1}^{(1)},k_{1}^{(2)}} \cdots \QStrict_{k_{n-1}^{(1)},k_{n-1}^{(2)}} \sum_{l_{1}=k_{1}^{(2)}}^{k_{2}^{(1)}} \cdots \hspace{-1ex} \sum_{l_{n-2}=k_{n-2}^{(2)}}^{k_{n-1}^{(1)}} \prod_{i=1}^{n-2}\fd_{l_i} \left[g(l_1,\dots,l_{n-2},k_{n-1}^{(2)}) \right]\right) \right|_{k_{i}^{(1)}=k_{i}^{(2)}=k_{i}} \\
		= (-1)^{n-1}\prod_{s=1}^{n-1} \Qid_{k_s} \left[g(k_1,\dots,k_{n-1})\right].
	\end{multline*}
	Again by exploiting telescoping sums, we obtain
	\begin{multline*}
		- \left.\left(\QStrict_{k_{1}^{(1)},k_{1}^{(2)}} \cdots \QStrict_{k_{n-1}^{(1)},k_{n-1}^{(2)}} \sum_{l_{1}=k_{1}^{(2)}}^{k_{2}^{(1)}} \cdots \hspace{-1ex} \sum_{l_{n-2}=k_{n-2}^{(2)}}^{k_{n-1}^{(1)}} \prod_{i=1}^{n-2}\fd_{l_i} \left[g(l_1,\dots,l_{n-2},k_{n-1}^{(2)})\right] \right) \right|_{k_{i}^{(1)}=k_{i}^{(2)}=k_{i}} \\
		\shoveleft = - \left.\left(\QStrict_{k_{1}^{(1)},k_{1}^{(2)}} \cdots \QStrict_{k_{n-2}^{(1)},k_{n-2}^{(2)}} \sum_{l_{1}=k_{1}^{(2)}}^{k_{2}^{(1)}} \cdots \hspace{-1ex} \sum_{l_{n-3}=k_{n-3}^{(2)}}^{k_{n-2}^{(1)}} \prod_{i=1}^{n-3}\fd_{l_i} \QStrict_{k_{n-1}^{(1)},k_{n-1}^{(2)}} \right.\right.\\
		\left.\left. \vphantom{\sum_{l_{n-3}=k_{n-3}^{(2)}}^{k_{n-2}^{(1)}}}   \left[g(l_1,\dots,l_{n-3},k_{n-1}^{(1)}+1,k_{n-1}^{(2)}) - g(l_1,\dots,l_{n-3},k_{(n-2)}^{(2)},k_{n-1}^{(2)})\right] \right) \right|_{k_{i}^{(1)}=k_{i}^{(2)}=k_{i}}.
	\end{multline*}
	Since $\QStrict_{k_{n-1}^{(1)},k_{n-1}^{(2)}} \left[g(l_1,\dots,l_{n-3},k_{n-1}^{(1)}+1,k_{n-1}^{(2)})\right]$ vanishes by assumption and $\QStrict_{x,y}$ reduces to $\Qid_y$ when applied to functions independent of $x$, the expression above simplifies to
	\begin{multline*}
		\left.\left(\QStrict_{k_{1}^{(1)},k_{1}^{(2)}} \cdots \QStrict_{k_{n-2}^{(1)},k_{n-2}^{(2)}} \Qid_{k_{n-1}}^{(2)} \sum_{l_{1}=k_{1}^{(2)}}^{k_{2}^{(1)}} \cdots \hspace{-1ex} \sum_{l_{n-3}=k_{n-3}^{(2)}}^{k_{n-2}^{(1)}} \prod_{i=1}^{n-3}\fd_{l_i} \right.\right.\\
		\left.\left. \vphantom{\sum_{l_{n-2}=k_{n-2}^{(2)}}^{k_{n-1}^{(1)}}}   \left[g(l_1,\dots,l_{n-3},k_{(n-2)}^{(2)},k_{n-1}^{(2)})\right] \right) \right|_{k_{i}^{(1)}=k_{i}^{(2)}=k_{i}}.
	\end{multline*}
	We complete the proof by repeating the last step $n-2$ times.
\end{proof}
Our task is now to apply these previous lemmata to the polynomials appearing in Theorem~\ref{thm:QHMTenumeration}. In order to do this, we define the operator $\T_{x,y} \coloneqq \E_x + \E_y - (2-Q)\E_x \E_y$, which is a simple modification of $\QStrict_{x,y}$.
\begin{lemma}\label{lem:AppSumOdd}
	The following operator formula holds:
	\begin{multline*}
		\Qsum{n-1}{n} \prod_{r=1}^{n-1} \Qid_{l_r} \prod_{1\leq s<t\leq n-1} \QStrict_{l_t,l_s} \T_{l_s,l_t} \left[ \det_{1\leq i,j \leq n-1} \left( \binom{l_i+i+j-b-\left(n-1\right)-2}{2j-1} \right) \right]\\	
		= \prod_{1\leq s<t\leq n} \QStrict_{k_t,k_s} \T_{k_s,k_t} \left[ \det_{1\leq i,j \leq n} \left( \binom{k_i+i+j-b-n-2}{2j-2} \right) \right].	
	\end{multline*}
\end{lemma}
\begin{proof}[Proof of Lemma~\ref{lem:AppSumOdd}]
	We want to use Lemma~\ref{lem:SumOpNormal}. Therefore, we set
	\begin{equation*}
		g(l_1,\dots,l_{n-1}) \coloneqq \prod_{r=1}^{n-1} \Qid_{l_r} \prod_{1\leq s<t\leq n-1} \QStrict_{l_t,l_s} \T_{l_s,l_t} \left[ \det_{1\leq i,j \leq n-1} \left( \binom{l_i+i+j-b-\left(n-1\right)-2}{2j} \right) \right].
	\end{equation*}
	Note that the operator
	\begin{equation*}
		\QStrict_{l_{i},l_{i+1}}\prod_{r=1}^{n-1} \Qid_{l_r} \prod_{1\leq s<t\leq n-1} \QStrict_{l_t,l_s} \T_{l_s,l_t}
	\end{equation*}
	is symmetric in $l_{i}$ and $l_{i+1}$ and that the polynomial
	\begin{equation*}
		\E_{l_{i}} \left[ \det_{1\leq i,j \leq n-1} \left( \binom{l_i+i+j-b-\left(n-1\right)-2}{2j} \right) \right]
	\end{equation*}
	is antisymmetric in $l_{i}$ and $l_{i+1}$. Consequently, the polynomial $\E_{l_{i}} \QStrict_{l_{i},l_{i+1}} \left[g(l_1,\dots,l_{n-1})\right]$ is also antisymmetric in $l_{i}$ and $l_{i+1}$ and, thus, divisible by the factor $l_{i+1}-l_{i}$. It follows that $g(l_1,\dots,l_{n-1})$ fulfils the requirements of Lemma~\ref{lem:SumOpNormal}. 
	
	The trick of the proof is the following observation: Suppose $f(x)$ is a function that is independent of $y$. Then the following operator expressions simplify: $\QStrict_{x,y} \left[f(x)\right] = ((Q-1)\E_x^{-1}+\id) \left[f(x)\right]$ and $\QStrict_{y,x} \left[f(x)\right] = \T_{x,y} \left[f(x)\right] = \T_{y,x} \left[f(x)\right] = \Qid_x \left[f(x)\right]$. By using the fact that $\fd_x \left[ \binom{x}{n} \right] = {\binom{x}{n-1}}$, we obtain
	\begin{multline*}
		\Qsum{n-1}{n} \prod_{r=1}^{n-1} \Qid_{l_r} \prod_{1\leq s<t\leq n-1} \QStrict_{l_t,l_s} \T_{l_s,l_t} \left[ \det_{1\leq i,j \leq n-1} \left( \binom{l_i+i+j-b-\left(n-1\right)-2}{2j-1} \right) \right]\\	
		\shoveleft = \Qsum{n-1}{n}\prod_{i=1}^{n-1}\fd_{l_i} \left[g(l_1,\dots,l_{n-1})\right]
		= \sum_{r=1}^{n} (-1)^{r-1}\prod_{s=1}^{r-1} \Qid_{k_s}\prod_{t=r+1}^{n} \QE_{k_t} \left(g(k_1,\dots ,\widehat{k_r},\dots ,k_{n})\right)\\
		\shoveleft = \sum_{r=1}^{n} (-1)^{r-1} \prod_{s=1}^{r-1} \Qid_{k_s} \prod_{t=r+1}^{n} \QE_{k_t} \prod_{\substack{u=1, \\ u \neq r}}^{n-1} \Qid_{k_u} \prod_{\substack{1\leq s<t\leq n-1, \\ s,t \neq r}} \QStrict_{k_t,k_s} \T_{k_s,k_t}\\
		\left[ \left. \det_{1\leq i,j \leq n-1} \left( \binom{l_i+i+j-b-\left(n-1\right)-2}{2j} \right) \right|_{(l_1,\dots,l_{n-1}) = (k_1,\dots ,\widehat{k_r},\dots ,k_{n})} \right]\\
		\shoveleft = \sum_{r=1}^{n} (-1)^{r-1}  \prod_{t=r+1}^{n} \E_{k_t}  \prod_{1\leq s<t\leq n-1} \QStrict_{k_t,k_s} \T_{k_s,k_t}\\
		\left[ \left. \det_{1\leq i,j \leq n-1} \left( \binom{l_i+i+j-b-\left(n-1\right)-2}{2j} \right) \right|_{(l_1,\dots,l_{n-1}) = (k_1,\dots ,\widehat{k_r},\dots ,k_{n})} \right].
	\end{multline*}
	In the last step, we use the fact that $\QE_{x} = \E_x ((Q-1)\E_x^{-1}+\id)$. Finally, we consider the determinant evaluation
	\begin{multline*}
		\det_{1\leq i,j \leq n} \left( \binom{k_i+i+j-b-n-2}{2j-2} \right)\\
		= \sum_{r=1}^{n} (-1)^{r-1} \prod_{t=r+1}^{n} \E_{k_t}  \left[ \left. \det_{1\leq i,j \leq n-1} \left( \binom{l_i+i+j-b-\left(n-1\right)-2}{2j} \right) \right|_{(l_1,\dots,l_{n-1}) = (k_1,\dots ,\widehat{k_r},\dots ,k_{n})} \right],
	\end{multline*}
	where we expand the determinant with respect to the first column. This completes the proof.
\end{proof}

\begin{lemma}\label{lem:AppSumEven}
	The following operator formula holds:
	\begin{multline*}
		\AQbsum{n} \prod_{1\leq s<t\leq n} \QStrict_{l_t,l_s} \T_{l_s,l_t} \left[ \det_{1\leq i,j \leq n} \left( \binom{l_i+i+j-b-n-2}{2j-2} \right) \right]\\	
		= (-1)^{n}\prod_{r=1}^{n} \Qid_{k_r} \prod_{1\leq s<t\leq n}  \QStrict_{k_t,k_s} \T_{k_s,k_t} \left[ \det_{1\leq i,j \leq n} \left( \binom{k_i+i+j-b-n-2}{2j-1} \right) \right].	
	\end{multline*}
\end{lemma}

\begin{proof}[Proof of Lemma~\ref{lem:AppSumEven}]
	This lemma is proved by means of Lemma~\ref{lem:SumOpAlt}. Therefore, we define
	\begin{equation*}
		g(b;l_1,\dots,l_{n}) \coloneqq \prod_{1\leq s<t\leq n} \QStrict_{l_t,l_s} \T_{l_s,l_t} \left[ \det_{1\leq i,j \leq n} \left( \binom{l_i+i+j-b-n-2}{2j-1} \right) \right].
	\end{equation*}
	Similarly as in the proof of Lemma~\ref{lem:AppSumOdd}, we can show that $g(b;l_1,\dots,l_{n})$ fulfils the conditions of Lemma~\ref{lem:SumOpAlt}. Consequently, we obtain
	\begin{multline*}
		\AQbsum{n} \prod_{1\leq s<t\leq n} \QStrict_{l_t,l_s} \T_{l_s,l_t} \left[ \det_{1\leq i,j \leq n} \left( \binom{l_i+i+j-b-n-2}{2j-2} \right) \right]\\
		=Q \sum_{r=1}^{n} (-1)^{r-1}\prod_{s=1}^{r-1} \Qid_{k_s}\prod_{t=r+1}^{n} \QE_{k_t} \left[g(b;k_1,\dots ,\widehat{k_r},\dots,k_{n},b+1)\right]\\
		+ (-1)^{n} \prod_{s=1}^{n} \Qid_{k_s} \left[g(b;k_1,\dots,k_{n})\right].
	\end{multline*}
	The proof reduces to showing that the first summand of the right-hand side vanishes. It suffices to establish that
	\begin{equation}\label{eq:AuxGZero}
		\left. g(b;l_1,\dots,l_{n}) \right|_{l_{n}=b+1}=0.
	\end{equation}
	We will use the following identity
	\begin{equation}\label{eq:AuxBinomial}
	\sum_{r=0}^{s}\binom{s}{r}(s-2r)^{2t+1}=0,
	\end{equation}
	which holds true for any nonnegative integers~$s$ and $t$. This follows from
	\begin{equation*}
		\sum_{r=0}^{s}\binom{s}{r}(s-2r)^{2t+1}
		=\sum_{r=0}^{s}\binom{s}{s-r}(s-2(s-r))^{2t+1}
		=-\sum_{r=0}^{s}\binom{s}{r}(s-2r)^{2t+1}.
	\end{equation*}
	Some manipulation yields
	\begin{equation*}
		\QStrict_{l_{n},l_i}\T_{l_i,l_{n}} = \left(\id-\left(4-2Q\right)\E_{l_i} + \left(5-4Q+Q^2\right)\E_{l_i}^2\right) + \E_{l_i} \left(\id-(2-Q)\E_{l_i}\right) \left(\E_{l_{n}} + \E_{l_{n}}^{-1}\right).
	\end{equation*}
	By Equation~\eqref{eq:HMTDetEven}, it is enough to show that we have the following identity for any nonnegative integer~$s$:
	\begin{equation}\label{eq:ComEquation}
	\left. \left( \E_{l_{n}} + \E_{l_{n}}^{-1} \right)^s \left[ \left(b+1-l_{n}\right) \prod_{i=1}^{n-1} \left( l_{n}-l_i+n-i \right) \left( 2b+n+2-l_{n}-l_i-i \right) \right] \right|_{l_{n}=b+1}=0.
	\end{equation}
	But since the left-hand side of \eqref{eq:ComEquation} is equivalent to
	\begin{multline*}
		\left. \sum_{r=0}^s \binom{s}{r} \E_{l_{n}}^{2r-s} \left[ \left(b+1-l_{n}\right) \prod_{i=1}^{n-1} \left( l_{n}-l_i+n-i \right) \left( 2b+n+2-l_{n}-l_i-i \right) \right] \right|_{l_{n}=b+1}\\
		= \sum_{r=0}^s \binom{s}{r} (s-2r) \prod_{i=1}^{n-1} \left( \left( b+n+1-l_i-i \right)^2 - (s-2r)^2 \right),
	\end{multline*}
	Equation~\eqref{eq:ComEquation} follows from \eqref{eq:AuxBinomial}. Consequently, Equation~\eqref{eq:AuxGZero} holds true.
	
\end{proof}

Finally, Theorem \ref{thm:QHMTenumeration} follows by a simple induction on $n$ using Lemmata~\ref{lem:AppSumOdd} and \ref{lem:AppSumEven} as well as Equations~\eqref{eq:HMTDetOdd} and \eqref{eq:HMTDetEven}.

\subsection{Proof of Theorem~\ref{thm:QHTREEenumeration}}

Fischer showed how to use the forward difference operator to enumerate truncated monotone triangles \cite{Fis11}. We generalise her ideas to the weighted enumeration of halved trees.

Note that the operator~$\Qfd_x$ is equivalent to $\fd_x \Qid_x^{-1}$. Since $\QStrict_{x,y}$ simplifies to $\Qid_y$ when applied to a function independent of $x$, we obtain by using \eqref{def:SumOp} that
\begin{multline*}
-\Qfd_{k_1} \Qsum{n-1}{n} f(l_1, l_2, \dots,l_{n-1})\\
\shoveleft = - \fd_{k_1} \Qid_{k_1}^{-1} \\
 \left. \times Q^{-1} \left( \Qid_{k_{1}^{(2)}} \QStrict_{k_{2}^{(1)},k_{2}^{(2)}} \cdots \QStrict_{k_{n}^{(1)},k_{n}^{(2)}}  \left[ \sum_{l_{1}=k_{1}^{(2)}}^{k_{2}^{(1)}} \cdots \hspace{-1ex} \sum_{l_{n-1}=k_{n-1}^{(2)}}^{k_{n}^{(1)}} f(l_1, \dots, l_{n-1}) \right] \right) \right|_{k_{i}^{(1)}=k_{i}^{(2)}=k_{i}}\\
= -\fd_{k_1} Q^{-1} \left.\left( \QStrict_{k_{2}^{(1)},k_{2}^{(2)}} \cdots \QStrict_{k_{n}^{(1)},k_{n}^{(2)}} \left[ \sum_{l_{1}=k_{1}^{(2)}}^{k_{2}^{(1)}} \cdots \hspace{-1ex} \sum_{l_{n-1}=k_{n-1}^{(2)}}^{k_{n}^{(1)}} f(l_1, \dots, l_{n-1}) \right] \right) \right|_{k_{i}^{(1)}=k_{i}^{(2)}=k_{i}}.
\end{multline*}
The next step follows from Fischer's crucial observation \cite{Fis11} that the application of the operator $-\fd_{k_1}$ has the effect of truncating the leftmost entry in the bottom row of the pattern and setting $k_1$ as the leftmost entry in the penultimate row. Thus, the previous expression is equal to
\begin{equation*}
Q^{-1} \left.\left( \QStrict_{k_{2}^{(1)},k_{2}^{(2)}} \cdots \QStrict_{k_{n}^{(1)},k_{n}^{(2)}} \left[ \sum_{l_{2} = k_{2}^{(2)}}^{k_{3}^{(1)}} \cdots \hspace{-1ex} \sum_{l_{n-1}=k_{n-1}^{(2)}}^{k_{n}^{(1)}} f(k_1,l_2, \dots, l_{n-1}) \right] \right) \right|_{k_{i}^{(1)}=k_{i}^{(2)}=k_{i}}.
\end{equation*}
We conclude that
\begin{multline*}
-\Qfd_{k_1} \Qsum{n-1}{n} f(l_1, l_2, \dots,l_{n-1})\\
= \sum_{\substack{ l_2 < \dots < l_{n-1}, \\ k_2 \leq l_2 \leq k_3 \leq \dots \leq k_{n-1} \leq l_{n-1} \leq k_n}} \hspace*{-1ex} Q^{[k_{2} < l_{2} < k_{3}]+\dots+[k_{n-1} < l_{n-1} < k_{n}]} f(k_1, l_2, \dots, l_{n-1}).
\end{multline*}
Hence, if we apply the operator $\prod_{r=1}^{\lceil n/2 \rceil} \left( -\Qfd_{k_r} \right)^{s_r}$, we truncate the $s_i$ bottom entries from the $i^{\text{th}}$ diagonal for each $1 \le i \le \lceil n/2 \rceil$.

\subsection{Proof of Theorem~\ref{thm:QHTREEConstantTerm}}

In the following proof, we present a method for transforming an operator formula into a constant term identity. This method can also be applied to other operator formulae that involve the same operands, for example Theorem~\ref{thm:VSASTPQCEnumeration}.

We assume that $n$ is odd. Then the following holds:
\begin{multline*}
	\prod_{r=1}^{\frac{n+1}{2}} ( -\Qfd_{k_r} )^{s_r} \prod_{1\leq s<t\leq \frac{n+1}{2}} \left(\E_{k_s}+\E_{k_t}^{-1}-(2-Q)\E_{k_s}\E_{k_t}^{-1}\right) \left(\E_{k_s}+\E_{k_t}-(2-Q)\E_{k_s}\E_{k_t}\right)\\
	\left[ \prod_{1\leq i<j\leq \frac{n+1}{2}} \frac{(k_j-k_i+j-i)(2b+n+2-k_i-k_j-i-j)}{(j-i)(i+j-1)} \right]\\
	\shoveleft = \ct_{x_1,\dots,x_{\frac{n+1}{2}}} \left( \prod_{r=1}^{\frac{n+1}{2}} \E_{x_r}^{k_r} ( -\Qfd_{x_r} )^{s_r}\right. \\
	\times \prod_{1\leq s<t\leq \frac{n+1}{2}} \left(\E_{x_s}+\E_{x_t}^{-1}-(2-Q)\E_{x_s}\E_{x_t}^{-1}\right) \left(\E_{x_s}+\E_{x_t}-(2-Q)\E_{x_s}\E_{x_t}\right) \\
	\left. \left[ \prod_{1\leq i<j\leq \frac{n+1}{2}} \frac{(x_j-x_i+j-i)(2b+n+2-x_i-x_j-i-j)}{(j-i)(i+j-1)} \right] \right).
\end{multline*}
Since $\fd_x=\E_x-\id$, the expression above is equal to
	\begin{multline}\label{eq:ConstantTerm1}
	\ct_{x_1,\dots,x_{\frac{n+1}{2}}} \left( \prod_{r=1}^{\frac{n+1}{2}} (\id + \fd_{x_r})^{k_r} \left(-\fd_{x_r}\left(Q\id-(1-Q)\fd_{x_r}\right)^{-1}\right)^{s_r} \right.\\
	\times \prod_{1\leq s<t\leq \frac{n+1}{2}} \left(\E_{x_t}^{-1} \left( Q\id + (Q-1)\fd_{x_s} + \fd_{x_t} + \fd_{x_s} \fd_{x_t} \right)\right.\\
	\left. \vphantom{\E_{x_t}^{-1}} \times \left( Q\id + (Q-1)\fd_{x_s} + (Q-1)\fd_{x_t} + (Q-2)\fd_{x_s}\fd_{x_t} \right)\right)\\
	\left. \left[ \prod_{1\leq i<j\leq \frac{n+1}{2}} \frac{(x_j-x_i+j-i)(2b+n+2-x_i-x_j-i-j)}{(j-i)(i+j-1)} \right] \right).
	\end{multline}
	Replacing $n$ by $(n+1)/2$ and then setting $k_i = x_i -b- n/2$ in Equation~\eqref{eq:DetFormulaeOdd} yields the following determinant identity:
	\begin{multline*}
	 \prod_{1\leq s<t\leq \frac{n+1}{2}} \E_{x_t}^{-1} \left[ \prod_{1\leq i<j\leq \frac{n+1}{2}}   \frac{(x_j-x_i+j-i)(2b+n+2-x_i-x_j-i-j)}{(j-i)(i+j-1)} \right]\\
	= \prod_{1\leq i<j\leq \frac{n+1}{2}} \frac{(x_j-x_i)(2b+n-x_j-x_i)}{(j-i)(j+i-1)}\\
	= (-1)^{\binom{\frac{n+1}{2}}{2}} \det_{1 \leq i,j \leq \frac{n+1}{2}} \left( \binom{x_i+j-b-\frac{n+3}{2}}{2j-2} \right).
	\end{multline*}
	By using the previous determinant evaluation and the Leibniz formula, we obtain that \eqref{eq:ConstantTerm1} is equal to
	\begin{multline}\label{eq:ConstantTerm2}
(-1)^{\binom{\frac{n+1}{2}}{2}} \ct_{x_1,\dots,x_{\frac{n+1}{2}}} \left( \sum_{\sigma \in \mathfrak{S}_{\frac{n+1}{2}}} \sgn (\sigma) \prod_{r=1}^{\frac{n+1}{2}} (\id + \fd_{x_r})^{k_r} \left(-\fd_{x_r}\left(Q\id-(1-Q)\fd_{x_r}\right)^{-1}\right)^{s_r}\right.\\
	\times \prod_{1\leq s<t\leq \frac{n+1}{2}} \left( Q\id + (Q-1)\fd_{x_s} + \fd_{x_t} + \fd_{x_s} \fd_{x_t} \right) \left( Q\id + (Q-1)\fd_{x_s} + (Q-1)\fd_{x_t} + (Q-2)\fd_{x_s}\fd_{x_t} \right)\\
	\left. \left[ \prod_{i=1}^{\frac{n+1}{2}} {\binom{x_i + \sigma(i) - b -\frac{n+3}{2}}{2\sigma(i) - 2}} \right] \right) \\
\shoveleft = (-1)^{\binom{\frac{n+1}{2}}{2}} \ct_{x_1,\dots,x_{\frac{n+1}{2}}}  \left( \sum_{\sigma \in \mathfrak{S}_{\frac{n+1}{2}}} \sgn (\sigma) \prod_{r=1}^{\frac{n+1}{2}} (\id + \fd_{x_r})^{k_r} \left(-\fd_{x_r}\left(Q\id-(1-Q)\fd_{x_r}\right)^{-1}\right)^{s_r}\right.\\
	\times \prod_{1\leq s<t\leq \frac{n+1}{2}} \left( Q\id + (Q-1)\fd_{x_s} + \fd_{x_t} + \fd_{x_s} \fd_{x_t} \right) \left( Q\id + (Q-1)\fd_{x_s} + (Q-1)\fd_{x_t} + (Q-2)\fd_{x_s}\fd_{x_t} \right)\\
	\left. \times \prod_{i=1}^{\frac{n+1}{2}} \left( \id + \fd_{x_i} \right)^{\sigma(i) - b -\frac{n+3}{2}} \left[ {\binom{x_i}{2\sigma(i) - 2}} \right] \right). 
	\end{multline}
	Since $\ct_x \left( \fd_x^s \left[ \binom{x}{t} \right] \right) = \delta_{s,t}$ holds true, \eqref{eq:ConstantTerm2} is equal to 
	\begin{multline}\label{eq:ConstantTerm3}
	(-1)^{\binom{\frac{n+1}{2}}{2}} \sum_{\sigma \in \mathfrak{S}_{\frac{n+1}{2}}} \sgn (\sigma) \left< x_1^{2\sigma(1) - 2} \cdots x_{\frac{n+1}{2}}^{2\sigma\left(\frac{n+1}{2}\right) - 2} \right> \left( \prod_{r=1}^{\frac{n+1}{2}} (1 + x_r)^{k_r + \sigma(r) - b - \frac{n+3}{2}} \left(-\frac{x_r}{Q-(1-Q) x_r}\right)^{s_r} \right. \\
	\left. \times \prod_{1\leq s<t\leq \frac{n+1}{2}} \left( Q + (Q-1) x_s + x_t + x_s x_t \right) \left( Q + (Q-1)x_s + (Q-1)x_t + (Q-2)x_s x_t \right) \right) \\
	= (-1)^{\binom{\frac{n+1}{2}}{2}} \ct_{x_1,\dots,x_{\frac{n+1}{2}}}  \left( \sum_{\sigma \in \mathfrak{S}_{\frac{n+1}{2}}} \sgn (\sigma) \prod_{r=1}^{\frac{n+1}{2}} x_r^{2-2\sigma(r)} (1 + x_r)^{k_r + \sigma(r) - b - \frac{n+3}{2}} \left(-\frac{x_r}{Q-(1-Q) x_r}\right)^{s_r} \right. \\
	\left. \times \prod_{1\leq s<t\leq \frac{n+1}{2}} \left( Q + (Q-1) x_s + x_t + x_s x_t \right) \left( Q + (Q-1)x_s + (Q-1)x_t + (Q-2)x_s x_t \right) \right),
	\end{multline}
	where $\left< x_1^{2\sigma(1)-2} \cdots x_{(n+1)/2}^{2\sigma((n+1)/2)-2}\right>$ denotes the coefficient of $x_1^{2\sigma(1)-2} \cdots x_{(n+1)/2}^{2\sigma((n+1)/2)-2}$ in the subsequent expression.

	By the generalised Vandermonde determinant evaluation \cite[Proposition 1]{Kra01}, the following identity holds true:
	\begin{multline*}
	\sum_{\sigma \in \mathfrak{S}_{\frac{n+1}{2}}} \sgn (\sigma) \prod_{r=1}^{\frac{n+1}{2}} x_r^{n+1-2\sigma(r)} \left( 1+x_r \right)^{\sigma(r) - 1} = \det_{1 \leq i,j \leq \frac{n+1}{2}} \left( x_i^{n+1-2j} \left( 1+x_i \right)^{j - 1} \right) \\
	= (-1)^{\binom{\frac{n+1}{2}}{2}} \prod_{1\leq s<t\leq \frac{n+1}{2}} \left( x_t - x_s \right) \left( x_s + x_t + x_s x_t \right).
	\end{multline*}
	Therefore, we can conclude that \eqref{eq:ConstantTerm3} is equal to
	\begin{multline*}
	\ct_{x_1,\dots,x_{\frac{n+1}{2}}} \left( \prod_{r=1}^{\frac{n+1}{2}} x_r^{1-n} (1+x_r)^{k_r -  b -\frac{n+1}{2}} \left(-\frac{x_r}{Q-(1-Q) x_r}\right)^{s_r}  \right.\\
	\times \prod_{1\leq s<t\leq \frac{n+1}{2}} \left( x_t - x_s\right) \left( x_s + x_t + x_s x_t \right)\\
	\left. \vphantom{\prod_{r=1}^{\frac{n+1}{2}}} \times \left( Q + (Q-1)x_s + x_t + x_s x_t \right) \left( Q + (Q-1)x_s + (Q-1)x_t + (Q-2)x_s x_t \right) \right).
	\end{multline*} 
	The case for even $n$ is treated similarly.

\subsection{Proof of Theorem~\ref{thm:HMTPQEnumeration}}

For the next proof, we essentially use the observation of Theorem~\ref{thm:QHTREEenumeration} that the application of the (generalised) forward difference operator has the effect of truncating entries of the diagonals.
	If the two bottommost entries in the $i^\text{th}$ diagonal of the halved tree are equal, we can truncate the bottommost entry of this diagonal which is reflected in the operator $-\Qfd_{k_i}$.
	
	If the two bottommost entries in the $i^\text{th}$ diagonal are not the same, we can count all halved trees and subtract those whose bottommost entries in the $i^\text{th}$ diagonal are equal. Thus, we need to apply the operator $\id - (-\Qfd_{k_i})$.

\subsection{Proof of Theorem~\ref{thm:VSASTQCEnumeration}}

As already noted in Section~\ref{sec:EnumAST}, Theorem~\ref{thm:VSASTQCEnumeration} follows from Theorem~\ref{thm:HMTPQEnumeration} using the bijection stated in Proposition~\ref{pro:Bijection} between vertically symmetric alternating sign trapezoids and halved trees.

\subsection{Proof of Theorem~\ref{thm:VSASTPQCEnumeration}}

To obtain the generating function of vertically symmetric alternating sign trapezoids, the key idea is to use Theorem~\ref{thm:VSASTQCEnumeration} and sum over all possible positions of 10-columns. Since
	\begin{multline*}
		Q^{-1} \bd_x \left(\id + \Qfd_x\right) = Q^{-1} \bd_x \left( \id + \fd_x \left(\id - (1-Q)\E_x\right)^{-1} \right)\\
		= Q^{-1} \bd_x \left(\left(\id-(1-Q)\E_x\right)+\left(\E_x-\id\right)\right)\left(\id - (1-Q)\E_x\right)^{-1} = \bd_x \E_x \left(\id - (1-Q)\E_x\right)^{-1} = \Qfd_x,
	\end{multline*} 
	we can manipulate the generating function~\eqref{eq:VSASTQCEnumeration} as follows:
	\begin{multline*}
		\prod_{c_i \in C_{10}} \left( -\Qfd_{c_i} \right) \prod_{\substack{1 \leq i \leq \frac{n}{2}, \\ c_i \notin C_{10}}} \left( \id + \Qfd_{c_i} \right) \prod_{r=1}^{\frac{n}{2}} \left( -\Qfd_{c_r} \right)^{-c_r-1} \QHMT_{n-1} \left( \frac{l-5}{2};\mathbf{c} \right)\\
		\shoveleft = \prod_{c_i \in C_{10}} \left( \id + \Qfd_{c_i} \right) \left( - Q^{-1} \bd_{c_i} \right) \prod_{\substack{1 \leq i \leq \frac{n}{2}, \\ c_i \notin C_{10}}} \left( \id + \Qfd_{c_i} \right)\\
		\times \prod_{r=1}^{\frac{n}{2}} \left( \left( \id + \Qfd_{c_i} \right) \left( - Q^{-1} \bd_{c_i} \right) \right)^{-c_r-1} \QHMT_{n-1} \left( \frac{l-5}{2};\mathbf{c} \right)\\ 
		= \prod_{c_i \in C_{10}} \left( - Q^{-1} \bd_{c_i} \right) \prod_{r=1}^{\frac{n}{2}} \left( \id + \Qfd_{c_i} \right)^{-c_r} \left( - Q^{-1} \bd_{c_i} \right)^{-c_r-1} \QHMT_{n-1} \left( \frac{l-5}{2};\mathbf{c} \right).
	\end{multline*}
	We want to sum over all possible positions of 10-columns. For this purpose, we make use of the elementary symmetric function: The generating function of vertically symmetric alternating sign trapezoids with $p$ many 10-columns is 
	\begin{equation*}
		e_p \left( - Q^{-1} \bd_{c_1}, \dots, - Q^{-1} \bd_{c_\frac{n}{2}} \right) \prod_{r=1}^{\frac{n}{2}} \left( \id + \Qfd_{c_r} \right)^{-c_r} \left( - Q^{-1} \bd_{c_r} \right)^{-c_r-1} \QHMT_{n-1} \left( \frac{l-5}{2};\mathbf{c} \right)
	\end{equation*}
	where $e_p$ denotes the $p^{\text{th}}$ elementary symmetric function. Since $e_p \left( x_1,\dots,x_m \right)$ is the coefficient of $P^p$ in $\prod_{i=1}^m \left( 1+ P x_i \right)$, the $PQ$-generating function is
	\begin{multline*}
		\prod_{r=1}^{\frac{n}{2}} \left( \id - P Q^{-1} \bd_{c_r} \right) \left(  \id + \Qfd_{c_r} \right)^{-c_r} \left( - Q^{-1} \bd_{c_r} \right)^{-c_r-1} \QHMT_{n-1} \left( \frac{l-5}{2};\mathbf{c} \right) \\
		= \prod_{r=1}^{\frac{n}{2}} \left( \id - (P-1)\Qfd_{c_r} \right) \left( - \Qfd_{c_r} \right)^{-c_r-1} \QHMT_{n-1} \left( \frac{l-5}{2};\mathbf{c} \right).
	\end{multline*}
	The transformation into a constant term identity is analogous to the proof of Theorem~\ref{thm:QHTREEConstantTerm}.

\subsection{Proof of Theorem~\ref{thm:VSASTPQEnumeration}}

Let $\mathfrak{S}_m$ be the symmetric group of degree $m$.
\begin{definition}
	The \emph{symmetriser} $\sym$ of a function $f(x_1,\dots,x_m)$ is defined as 
	\begin{equation*}
	\sym_{x_1,\dots,x_m} f(x_1,\dots,x_m) \coloneqq \sum_{\sigma\in\mathfrak{S}_m} f(x_{\sigma(1)},\dots,x_{\sigma(m)});
	\end{equation*}
	the \emph{antisymmetriser} $\asym$ is given by
	\begin{equation*}
	\asym_{x_1,\dots,x_m} f(x_1,\dots,x_m) \coloneqq \sum_{\sigma\in\mathfrak{S}_m} \sgn(\sigma) f(x_{\sigma(1)},\dots,x_{\sigma(m)}).
	\end{equation*}
\end{definition}

We use the symmetriser in combination with the following lemma, which Zeilberger called the Stanton-Stembridge trick \cite[Crucial Fact $\aleph_4$]{Zei96a}.

\begin{lemma}
	For a formal Laurent series $f(x_1,\dots,x_m)$ and a permutation $\sigma \in \mathfrak{S}_m$,
	\begin{equation*}
	\ct_{x_1,\dots,x_m} \left(f(x_{\sigma(1)},\dots,x_{\sigma(m)})\right)=\ct_{x_1,\dots,x_m} \left(f(x_1,\dots,x_m)\right).
	\end{equation*}
\end{lemma}

As a consequence, it follows that
\begin{equation}\label{eq:SymMethod}
\ct_{x_1,\dots,x_m} \left(f(x_1,\dots,x_m)\right) = \ct_{x_1,\dots,x_m} \left( \frac{1}{m!} \sym_{x_1,\dots,x_m} f(x_1,\dots,x_m) \right).
\end{equation}
To prove Theorem~\ref{thm:VSASTPQEnumeration}, we need the following lemma. It is a generalisation of \cite[Lemma 9]{Fis19}.
\begin{lemma}\label{lem:QASym}
The following identity holds true:
\begin{multline}\label{eq:QASym}
		\asym_{x_1,\dots,x_m} \left( \prod_{r=1}^{m} \frac{\left( \frac{x_r(1+x_r)}{Q+x_r} \right)^{r-1}}{1-\prod_{j=r}^{m} \frac{x_j(1+x_j)}{Q+x_j} } \prod_{1\leq s<t\leq m} ( Q + (Q-1) x_s + x_t + x_s x_t ) \right)\\
		= \prod_{r=1}^{m} \frac{Q+x_r}{Q - x_r^2} \prod_{1\leq s<t\leq m} \frac{(Q(1+x_s)(1+x_t)-x_s x_t)(x_t-x_s)}{Q-x_s x_t}.
\end{multline}	
\end{lemma}
\begin{proof}[Proof of Lemma~\ref{lem:QASym}]
	We show the identity by induction on $m$; it is proved in a similar way as \cite[Lemma 9]{Fis19} which follows from \eqref{eq:QASym} by setting $Q=1$. 
	
	The base case $m=1$ is clear. We set $f_m(x_1,\dots,x_m)$ to be the argument of the antisymmetriser on the left-hand side of \eqref{eq:QASym}; that is,
	\begin{equation*}
		f_m(x_1,\dots,x_m) \coloneqq \prod_{r=1}^{m} \frac{\left( \frac{x_r(1+x_r)}{Q+x_r} \right)^{r-1}}{1-\prod_{j=r}^{m} \frac{x_j(1+x_j)}{Q+x_j} } \prod_{1\leq s<t\leq m} ( Q + (Q-1) x_s + x_t + x_s x_t ).
	\end{equation*}
	It can readily be seen that $f_m(x_1,\dots,x_m)$ is equal to
	\begin{multline*}
	\left( \prod_{r=2}^m ( Q + (Q-1) x_1 + x_r + x_1 x_r ) \frac{x_r(1+x_r)}{Q+x_r} \right) \left( 1-\prod_{r=1}^m \frac{x_r(1+x_r)}{Q+x_r} \right)^{-1} f_{m-1}(x_2,\dots,x_m).
	\end{multline*}
	By the definition of the antisymmetriser we see that $f_m$ satisfies the following recursion:
	\begin{multline*}
		\asym_{x_1,\dots,x_m} f_{m}(x_1,\dots,x_m) = \left( 1-\prod_{r=1}^m \frac{x_r(1+x_r)}{Q+x_r} \right)^{-1}\\
		\times \sum_{k=1}^m \left( (-1)^{k+1} \prod_{\substack{l=1, \\ l \neq k}}^{m} \frac{x_l(1+x_l)}{Q+x_l} (Q+(Q-1)x_k+x_l+x_k x_l) \right.\\
		\left. \times \vphantom{\prod_{\substack{l=1, \\ l \neq k}}^{m}} \asym_{x_1,\dots,\widehat{x_k},\dots,x_m} f_{m-1}(x_1,\dots,\widehat{x_k},\dots,x_m) \right).	
	\end{multline*}
	We want to show that the right-hand side of \eqref{eq:QASym} fulfils the same recursion. Some manipulation yields that this is equivalent to proving the following polynomial identity:
	\begin{multline*}
		\left( \prod_{r=1}^m (Q+x_r) - \prod_{r=1}^m x_r(1+x_r) \right) \prod_{1 \le s < t \le m} \left( Q(1+x_s)(1+x_t)-x_s x_t \right)\\
		= \sum_{k=1}^m (Q-x_k^2) \prod_{\substack{1 \le s < t \le m, \\ s,t \neq k}} \left( Q(1+x_s)(1+x_t)-x_s x_t \right)\\
		\times \prod_{\substack{1 \le r \le m, \\ r \neq k}} \frac{x_r(1+x_r)(Q-x_k x_r)(Q+(Q-1)x_k+x_r+x_k x_r)}{x_r-x_k}.
	\end{multline*}
	Both sides are symmetric polynomials in ${x_1,\dots,x_m}$, and the leading terms are $-(Q-1)^{\binom{m}{2}} x_1^{m+1} \cdots \allowbreak x_m^{m+1}$. The identity holds true for the evaluations $x_i=0$ and $x_i=-1$, and both sides vanish for $x_i = Q(1+x_j)/((1-Q)x_j-Q)$ for all $i \neq j$. This completes the proof of \eqref{eq:QASym}.
\end{proof}

We are now in a position to prove Theorem~\ref{thm:VSASTPQEnumeration}. To this end, we change the number of the columns from $-(n-2)$ to $0$ instead of from $-(n-1)$ to $-1$; that is, we shift $c_r \mapsto c_r-1$. Then the generating function  of all halved vertically symmetric alternating sign trapezoids with prescribed $1$-columns is equal to the constant term in $x_1$, \dots, $x_{n/2}$ of
	\begin{multline*}
		\prod_{r=1}^{\frac{n}{2}} x_r^{2-n} \left(1+x_r\right)^{c_r-\frac{l-3}{2}-\frac{n}{2}} \left( \frac{Q + (Q-P)x_r}{Q - (1-Q)x_r} \right) \left(-\frac{x_r}{Q - (1-Q)x_r}\right)^{-c_r} \\
		\times \prod_{1\leq s<t\leq \frac{n}{2}} \left( \left( x_t - x_s\right) \left( x_s + x_t + x_s x_t \right) \right.\\
		\left. \times \left( Q + (Q-1)x_s + x_t + x_s x_t \right) \left( Q + (Q-1)x_s + (Q-1)x_t + (Q-2)x_s x_t \right) \right).
	\end{multline*}
	We sum over all possible $1$-column vectors~$\mathbf{c}$ and obtain the constant term in $x_1$, \dots, $x_{n/2}$ of 
	\begin{multline*}
		\sum_{c_1<\dots<c_{\frac{n}{2}}\leq 0} \Bigg( \prod_{r=1}^{\frac{n}{2}} x_r^{2-n} \left(1+x_r\right)^{c_r-\frac{l-3}{2}-\frac{n}{2}} \left( \frac{Q + (Q-P)x_r}{Q - (1-Q)x_r} \right) \left(-\frac{x_r}{Q - (1-Q)x_r}\right)^{-c_r} \\
		\times \prod_{1\leq s<t\leq \frac{n}{2}} \left( \left( x_t - x_s\right) \left( x_s + x_t + x_s x_t \right) \right. \\
		\left. \times \left( Q + (Q-1)x_s + x_t + x_s x_t \right) \left( Q + (Q-1)x_s + (Q-1)x_t + (Q-2)x_s x_t \right) \right) \Bigg),
	\end{multline*}
	which is equal to the constant term in $x_1$, \dots, $x_{n/2}$ of
	\begin{multline}\label{prf:QConstantTerm}
		Q^{-\frac{n}{2}} \prod_{r=1}^{\frac{n}{2}} \frac{x_r^{2-n} (1+x_r)^{-\frac{l-3}{2}-\frac{n}{2}} \left( \frac{Q + (Q-P)x_r}{Q - (1-Q)x_r} \right) \left( \frac{-x_r}{(Q-(1-Q)x_r)(1+x_r)} \right)^{\frac{n}{2}-r}}{1-\prod_{j=1}^{r} \left(\frac{-x_j}{(Q-(1-Q)x_j)(1+x_j)}\right)} \\
		\times \prod_{1\leq s<t\leq \frac{n}{2}} \left( \left( x_t - x_s\right) \left( x_s + x_t + x_s x_t \right) \right. \\
		\left. \times \left( Q + (Q-1)x_s + x_t + x_s x_t \right) \left( Q + (Q-1)x_s + (Q-1)x_t + (Q-2)x_s x_t \right) \right) 
	\end{multline}
	because of the following geometric series expression
	\begin{equation*}
		\sum_{c_1<\dots<c_m\leq 0}y_1^{-c_1} \cdots y_{m}^{-c_m} = \prod_{r=1}^{m} \frac{y_r^{m-r}}{1-\prod_{j=1}^{r} y_j}
	\end{equation*}
	with $y_r=-x_r/((Q-(1-Q)x_r)(1+x_r))$ for all $1 \leq r \leq n/2$.
	
	We set $m=n/2$ in \eqref{eq:SymMethod} and apply it to \eqref{prf:QConstantTerm}. Thus, we obtain the constant term in $x_1$, \dots, $x_{n/2}$ of
	\begin{multline*}
	\frac{1}{\left(\frac{n}{2}\right)!} \sym_{x_1,\dots,x_{\frac{n}{2}}} \left( \prod_{r=1}^{\frac{n}{2}} \frac{x_r^{2-n} (1+x_r)^{-\frac{l-3}{2}-\frac{n}{2}} \left( \frac{Q + (Q-P)x_r}{Q - (1-Q)x_r} \right) \left( \frac{-x_r}{(Q-(1-Q)x_r)(1+x_r)} \right)^{\frac{n}{2}-r}}{1-\prod_{j=1}^{r} \left(\frac{-x_j}{(Q-(1-Q)x_j)(1+x_j)}\right)} \right.\\
	\times \prod_{1\leq s<t\leq \frac{n}{2}} \left( \left( x_t - x_s\right) \left( x_s + x_t + x_s x_t \right) \right. \\
	\left. \times 
	\vphantom{\frac{x_r^{2-n} (1+x_r)^{-\frac{l-3}{2}-\frac{n}{2}} \left( \frac{Q + (Q-P)x_r}{Q - (1-Q)x_r} \right) \left( \frac{-x_r}{(Q-(1-Q)x_r)(1+x_r)} \right)^{\frac{n}{2}-r}}{1-\prod_{j=1}^{r} \left(\frac{-x_j}{(Q-(1-Q)x_j)(1+x_j)}\right)}}
	\left. \left( Q + (Q-1)x_s + x_t + x_s x_t \right) \left( Q + (Q-1)x_s + (Q-1)x_t + (Q-2)x_s x_t \right) \right) \right),
	\end{multline*}
	which is equal to
	\begin{multline}\label{eq:QASymConstLastStep}		
	\frac{1}{\left(\frac{n}{2}\right)!} \prod_{r=1}^{\frac{n}{2}} \frac{ \frac{Q + (Q-P)x_r}{Q - (1-Q)x_r} }{x_r^{n-2}  \left(1+x_r\right)^{\frac{l-3}{2}+\frac{n}{2}}} \\
	\times \prod_{1\leq s<t\leq \frac{n}{2}} \left( x_t - x_s\right) \left( x_s + x_t + x_s x_t \right) \left( Q + (Q-1)x_s + (Q-1)x_t + (Q-2)x_s x_t \right)\\		
	\times \asym_{x_1,\dots,x_{\frac{n}{2}}} \left( \prod_{r=1}^{\frac{n}{2}} \frac{\left( \frac{-x_r}{(Q-(1-Q)x_r)(1+x_r)} \right)^{\frac{n}{2}-r}}{1-\prod_{j=1}^{r} \left( \frac{-x_j}{(Q-(1-Q)x_j)(1+x_j)} \right)} \prod_{1\leq s<t\leq \frac{n}{2}} \left( Q + (Q-1)x_s + x_t + x_s x_t \right) \right).
	\end{multline}
	To simplify the previous expression, we use Lemma~\ref{lem:QASym}. We replace $x_i \mapsto -x_{m+1-i}/(1+x_{m+1-i})$ in Equation~\eqref{eq:QASym} to obtain the following identity:
	\begin{multline}\label{eq:QASymVar}
		\asym_{x_1,\dots,x_m} \left( \prod_{r=1}^{m} \frac{\left( \frac{-x_r}{(Q-(1-Q)x_r)(1+x_r)} \right)^{m-r}}{1-\prod_{j=1}^{r} \left( \frac{-x_j}{(Q-(1-Q)x_j)(1+x_j)} \right)} \prod_{1\leq s<t\leq m} \left( Q + (Q-1) x_s + x_t + x_s x_t \right) \right)\\
		= \prod_{r=1}^{m} \frac{(1+x_r)(Q-(1-Q)x_r)}{Q(1+x_r)^2 - x_r^2} \prod_{1\leq s<t\leq m} \frac{\left(Q-x_s x_t\right) \left(x_t - x_s\right)}{Q(1+x_s)(1+x_t)-x_s x_t}.
	\end{multline}
	To complete the proof of Theorem~\ref{thm:VSASTPQEnumeration}, we finally apply Equation~\eqref{eq:QASymVar} to \eqref{eq:QASymConstLastStep}.	

\appendix

\section{The \texorpdfstring{$2$}{2}-Enumeration of Halved Monotone Triangles}
\label{sec:2enum}

By setting $Q=1$ in the generating function $\QHMT_n\left(b;\mathbf{k}\right)$, we recover the straight enumeration of halved monotone triangles. The \emph{$2$-enumeration} is obtained by setting $Q=2$. It turns out that the $2$-enumeration of halved monotone triangles can be written in an operator-free product formula. This comes as no surprise: The $2$-enumeration of alternating sign matrices had already been solved by Mills, Robbins, and Rumsey \cite{MRR83} whereas the straight $1$-enumeration remained unsolved for over a decade. Lai investigates the $2$-enumeration of so-called \emph{antisymmetric monotone triangles} \cite{Lai} which has apparently been proved before by Jokusch and Propp in an unpublished work. Antisymmetric monotone triangles essentially correspond to halved monotone triangles with no entry larger than $-1$. Lai considers the following \emph{$q$-weight}: It counts all entries that appear in some row but not in the row directly above it. We can recover the $q$-weight of a halved monotone triangle from its $Q$-weight: Consider two consecutive rows of a halved monotone triangle and suppose that the upper row contributes the weight~$Q^m$. If the lower row has the same number of entries, then this row below contributes the weight~$q^m$; if the lower row has one entry more, it contributes the weight~$q^{m+1}$. The top row of a halved monotone triangle always contributes a factor~$q$. In total, a halved monotone triangle of order~$n$ and $Q$-weight~$Q^m$ has $q$-weight~$q^{m+\lfloor n/2 \rfloor}$. This observation implies the following theorem as a corollary of \cite[Theorem 1.1]{Lai}:

\begin{theorem}
	The $2$-enumeration of halved monotone triangles of order~$n$ with prescribed bottom row~$(k_1,\dots,\allowbreak k_{\lceil n/2 \rceil})$ and no entry larger than $b$ is given by
	\begin{equation*}
		4^{\binom{\frac{n+1}{2}}{2}} \prod_{1\le i<j\le \frac{n+1}{2}}\frac{(k_j-k_i)(2b+1-k_i-k_j)}{(j-i)(i+j-1)} 
	\end{equation*}
	if $n$ is odd and by	
	\begin{equation*}
		4^{\binom{\frac{n}{2}}{2}} \prod_{1\le i<j\le \frac{n}{2}}\frac{(k_j-k_i)(2b+1-k_i-k_j)}{(j-i)(i+j)} \prod_{i=1}^{\frac{n}{2}}\frac{2b+1-2k_i}{i}
	\end{equation*}
	if $n$ is even.
\end{theorem}

\section{Enumeration of Halved Gelfand-Tsetlin Patterns}
\label{sec:HGTP}

If we weaken the conditions in the definition of halved monotone triangles by allowing rows to be weakly increasing, we obtain \emph{halved Gelfand-Tsetlin patterns}. They are enumerated by the operands in Theorem~\ref{thm:QHMTenumeration}. First, we derive an enumeration formula by means of nonintersecting lattice paths. Then, we encounter two more interpretations of halved Gelfand-Tsetlin patterns, one as lozenge tilings of certain regions and one in terms of representations of the symplectic group.

To enumerate halved Gelfand-Tsetlin patterns by nonintersecting lattice paths, we modify the bijection in~\cite{Fis12} between regular Gelfand-Tsetlin patterns and nonintersecting lattice paths: Given a halved Gelfand-Tsetlin pattern with $n$ rows, bottom row~$\left( k_1, \dots, k_{\lceil n/2 \rceil} \right)$ and no entry larger than $b$, we divide it into $\nearrow$- diagonals and number them from left to right by $1$ to $\lceil n/2 \rceil$. At the right end of each diagonal we put an additional bounding entry~$b$. We add $i-1$ to the entries of the $i^{\text{th}}$ diagonal. These entries are the heights of paths connecting the starting points $\left( -1,k_i+i-1 \right)$ and end points $\left(n+2-2i,b+i-1\right)$ with $(1,0)$- and $(0,1)$-steps. We cut off the first and the last step of each path since they are always horizontal steps, and thus we obtain the starting points $S_i \coloneqq \left( 0,k_i+i-1 \right)$ and the end points $E_i \coloneqq \left( n+1-2i,b+i-1 \right)$.

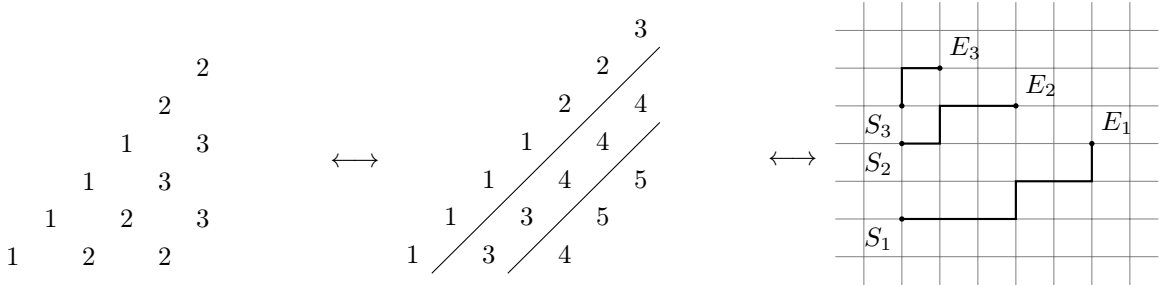
\begin{figure}[htb]
	\centering
		\raisebox{1ex}{\begin{tikzpicture}[scale=.5]
			\hmt{{1,2,2},{1,2,3},{1,3},{1,3},{2},{2}}
			\node at (9,2.5){$\longleftrightarrow$};
			\end{tikzpicture}
			\begin{tikzpicture}[scale=.5]
			\hmt{{1,3,4,},{1,3,5},{1,4,5},{1,4},{2,4},{2},{3}}
			\draw (0.5,-0.5) -- (6.5,5.5);
			\draw (2.5,-0.5) -- (6.5,3.5);
			\node at (10,2.5){$\longleftrightarrow$};
			\end{tikzpicture}}
		\begin{tikzpicture}[scale=.5]
		\draw [help lines,step=1cm] (-1.75,-.75) grid (6.75,6.75);
		
		\coordinate[label=225:$S_1$] (S1) at (0,1);
		\fill (S1) circle (2pt);
		
		\coordinate[label=225:$S_2$] (S2) at (0,3);
		\fill (S2) circle (2pt);
		
		\coordinate[label=225:$S_3$] (S3) at (0,4);
		\fill (S3) circle (2pt);
		
		\coordinate[label=45:$E_1$] (E1) at (5,3);
		\fill (E1) circle (2pt);
		
		\coordinate[label=45:$E_2$] (E2) at (3,4);
		\fill (E2) circle (2pt);
		
		\coordinate[label=45:$E_3$] (E3) at (1,5);
		\fill (E3) circle (2pt);
		
		\draw[thick] (S1) -- (3,1) -- (3,2) -- (5,2) -- (E1);
		
		\draw[thick] (S2) -- (1,3) -- (1,4) -- (E2);
		
		\draw[thick] (S3) -- (0,5) -- (E3);
		
		\end{tikzpicture}
		\caption{Transforming a halved Gelfand-Tsetlin pattern with no entry larger than $3$ into a family of nonintersecting lattice paths}
		\label{figure:LatticePaths}
\end{figure}

This yields a bijection between halved Gelfand-Tsetlin patterns and families of certain nonintersecting lattice paths; see Figure~\ref{figure:LatticePaths} for an example. Due to the Lindström-Gessel-Viennot theorem, the number of the nonintersecting lattice paths connecting the given start and end points is given by
\begin{equation*}\label{eq:LatticPathsDet}
	\det\limits_{1 \le i,j \le \lceil \frac{n}{2} \rceil}\left( \binom{n+b+1-k_i-i-j}{n+1-2j} \right).
\end{equation*}
By using \cite[Theorem~27]{Kra01}, this determinant evaluates to 
\begin{equation*}
	\prod_{i=1}^{\lceil \frac{n}{2} \rceil} \frac{(b+n-\lceil \frac{n}{2} \rceil +1 -k_i -i)!}{(b+ \lceil \frac{n}{2} \rceil -k_i -i)! (n+1-2i)!} \prod_{1\leq i<j\leq \lceil \frac{n}{2} \rceil} (k_j-k_i+j-i) (2b+n+2 -k_i-k_j -i-j),
\end{equation*}
which is equivalent to \eqref{eq:HMTDetEven} or \eqref{eq:HMTDetOdd} if $n$ is even or odd, respectively.

Halved Gelfand-Tsetlin patterns can also be interpreted as lozenge tilings of so-called \emph{quartered hexagons}: Consider a trapezoidal region in the triangular lattice as displayed in Figure~\ref{figure:QuarteredHexagons}; its left side has length~$s$, its upper and lower parallel sides have length~$t$ and $t+\lceil s/2 \rceil$, respectively, and its right side is a vertical zigzag line of length~$s$. Remove $\lceil s/2 \rceil$ unit triangles from the lower side at positions $d_1, d_2,\dots,d_{\lceil s/2 \rceil}$. We denote this region by $H_{s,t} (d_1, d_2,\dots,d_{\lceil s/2 \rceil})$. 
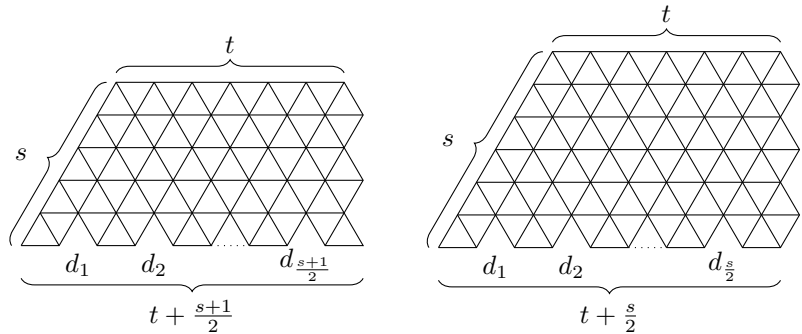
\begin{figure}[htb]
	\centering
		\begin{tikzpicture}[scale=.5]
		
		
		\coordinate (1) at (0,0) {};
		\coordinate (2) at ($ (1) + (1,0) $) {};
		\coordinate (3) at ($ (2) + (1,0) $) {};
		\coordinate (4) at ($ (3) + (1,0) $) {};
		\coordinate (5) at ($ (4) + (1,0) $) {};
		\coordinate (6) at ($ (5) + (1,0) $) {};
		\coordinate (7) at ($ (6) + (1,0) $) {};
		\coordinate (8) at ($ (7) + (1,0) $) {};
		\coordinate (9) at ($ (8) + (1,0) $) {};
		\coordinate (10) at ($ (9) + (1,0) $) {};
		
		\coordinate (11) at ($ (1) + (60:1) $) {};
		\coordinate (12) at ($ (11) + (1,0) $) {};
		\coordinate (14) at ($ (12) + (2,0) $) {};
		\coordinate (18) at ($ (14) + (4,0) $) {};
		\coordinate (19) at ($ (18) + (1,0) $) {};
		
		\coordinate (20) at ($ (11) + (60:1) $) {};
		\coordinate (28) at ($ (20) + (8,0) $) {};
		
		\coordinate (29) at ($ (20) + (60:1) $) {};
		\coordinate (36) at ($ (29) + (7,0) $) {};
		
		\coordinate (37) at ($ (29) + (60:1) $) {};
		\coordinate (44) at ($ (37) + (7,0) $) {};
		
		\coordinate (45) at ($ (37) + (60:1) $) {};
		\coordinate (46) at ($ (45) + (1,0) $) {};
		\coordinate (47) at ($ (46) + (1,0) $) {};
		\coordinate (48) at ($ (47) + (1,0) $) {};
		\coordinate (49) at ($ (48) + (1,0) $) {};
		\coordinate (50) at ($ (49) + (1,0) $) {};
		\coordinate (51) at ($ (50) + (1,0) $) {};
		
		
		\draw (1) -- (2) -- (12) -- (3) -- (4) -- (14) -- (5) -- (6);
		\draw [dotted] (6) -- (7);
		\draw (7) -- (8) -- (18) -- (9) -- (10) -- (19) -- (28) -- (36) -- (44) -- (51) -- (45) -- (1);
		
		
		\node at ($ (2) + (.5,-.5) $) {$d_1$};
		\node at ($ (4) + (.5,-.5) $) {$d_2$};
		\node at ($ (8) + (.5,-.5) $) {$d_{\frac{s+1}{2}}$};
		
		
		\draw (12) -- (46);
		\draw (3) -- (47);
		\draw (14) -- (48);
		\draw (5) -- (49);
		\draw (6) -- (50);
		\draw (7) -- (51);
		\draw (18) -- (36);
		\draw (9) -- (19);
		
		
		\draw (2) -- (11);
		\draw (12) -- (20);
		\draw (4) -- (29);
		\draw (14) -- (37);
		\draw (6) -- (45);
		\draw (7) -- (46);
		\draw (8) -- (47);
		\draw (18) -- (48);
		\draw (19) -- (49);
		\draw (36) -- (50);
		
		
		\draw (11) --++(8,0);
		\draw (20) --++(8,0);
		\draw (29) --++(7,0);
		\draw (37) --++(7,0);
		
		
		\node at ($ (48) + (0,1) $) {$t$};
		
		\draw [decorate,decoration={brace,amplitude=5pt}] (2.5,4.580) -- (8.5,4.580);
		
		\node at ($ (1,2.165) + (-1,0.25) $) {$s$};
		
		\draw [decorate,decoration={brace,amplitude=5pt},rotate around={60:(1,2.165)},=60] (-1.5,2.165) -- (3.5,2.165);
		
		\node at ($ (4.5,-0.866) + (0,-1) $) {$t+\frac{s+1}{2}$};
		
		\draw [decorate,decoration={brace,amplitude=5pt},rotate around={180:(4.5,-0.866)},=60] (0,-0.866) -- (9,-0.866);
		\end{tikzpicture}
		\qquad
		\begin{tikzpicture}[scale=.5]
		
		
		\coordinate (1) at (0,0) {};
		\coordinate (2) at ($ (1) + (1,0) $) {};
		\coordinate (3) at ($ (2) + (1,0) $) {};
		\coordinate (4) at ($ (3) + (1,0) $) {};
		\coordinate (5) at ($ (4) + (1,0) $) {};
		\coordinate (6) at ($ (5) + (1,0) $) {};
		\coordinate (7) at ($ (6) + (1,0) $) {};
		\coordinate (8) at ($ (7) + (1,0) $) {};
		\coordinate (9) at ($ (8) + (1,0) $) {};
		\coordinate (10) at ($ (9) + (1,0) $) {};
		
		\coordinate (11) at ($ (1) + (60:1) $) {};
		\coordinate (12) at ($ (11) + (1,0) $) {};
		\coordinate (14) at ($ (12) + (2,0) $) {};
		\coordinate (18) at ($ (14) + (4,0) $) {};
		\coordinate (20) at ($ (18) + (2,0) $) {};
		
		\coordinate (21) at ($ (11) + (60:1) $) {};
		\coordinate (29) at ($ (21) + (8,0) $) {};
		
		\coordinate (30) at ($ (21) + (60:1) $) {};
		\coordinate (38) at ($ (30) + (8,0) $) {};
		
		\coordinate (39) at ($ (30) + (60:1) $) {};
		\coordinate (46) at ($ (39) + (7,0) $) {};
		
		\coordinate (47) at ($ (39) + (60:1) $) {};
		\coordinate (54) at ($ (47) + (7,0) $) {};
		
		\coordinate (55) at ($ (47) + (60:1) $) {};
		\coordinate (56) at ($ (55) + (1,0) $) {};
		\coordinate (57) at ($ (56) + (1,0) $) {};
		\coordinate (58) at ($ (57) + (1,0) $) {};
		\coordinate (59) at ($ (58) + (1,0) $) {};
		\coordinate (60) at ($ (59) + (1,0) $) {};
		\coordinate (61) at ($ (60) + (1,0) $) {};
		
		
		\draw (1) -- (2) -- (12) -- (3) -- (4) -- (14) -- (5) -- (6);
		\draw [dotted] (6) -- (7);
		\draw (7) -- (8) -- (18) -- (9) -- (10) -- (20) -- (29) -- (38) -- (46) -- (54) -- (61) -- (55) -- (1);
		
		
		\node at ($ (2) + (.5,-.5) $) {$d_1$};
		\node at ($ (4) + (.5,-.5) $) {$d_2$};
		\node at ($ (8) + (.5,-.5) $) {$d_{\frac{s}{2}}$};
		
		
		\draw (12) -- (56);
		\draw (3) -- (57);
		\draw (14) -- (58);
		\draw (5) -- (59);
		\draw (6) -- (60);
		\draw (7) -- (61);
		\draw (18) -- (46);
		\draw (9) -- (29);
		
		
		\draw (2) -- (11);
		\draw (12) -- (21);
		\draw (4) -- (30);
		\draw (14) -- (39);
		\draw (6) -- (47);
		\draw (7) -- (55);
		\draw (8) -- (56);
		\draw (18) -- (57);
		\draw (10) -- (58);
		\draw (29) -- (59);
		\draw (46) -- (60);
		
		
		\draw (11) --++(9,0);
		\draw (21) --++(8,0);
		\draw (30) --++(8,0);
		\draw (39) --++(7,0);
		\draw (47) --++(7,0);
		
		
		\node at ($ (58) + (0,1) $) {$t$};
		
		\draw [decorate,decoration={brace,amplitude=5pt}] (3,5.446) -- (9,5.446);
		
		\node at ($ (1.25,2.598) + (-1,0.25) $) {$s$};
		
		\draw [decorate,decoration={brace,amplitude=5pt},rotate around={60:(1.25,2.598)},=60] (-1.75,2.598) -- (4.25,2.598);
		
		\node at ($ (4.5,-0.866) + (0,-1) $) {$t+\frac{s}{2}$};
		
		\draw [decorate,decoration={brace,amplitude=5pt},rotate around={180:(4.5,-0.866)},=60] (0,-0.866) -- (9,-0.866);
		\end{tikzpicture}
		\caption{The regions $H_{s,t} (d_1, d_2,\dots,d_{\lceil s/2 \rceil})$ for $s$ odd (left) and even (right)}
		\label{figure:QuarteredHexagons}
\end{figure}

Halved Gelfand-Tsetlin patterns with $n$ rows, bottom row~$(k_1, k_2,\dots,k_{\lceil n/2 \rceil})$ and no entry larger than $b$ correspond to lozenge tilings of the region $H_{n,b-k_1} (1 , k_2+2-k_1 , \dots , k_{\lceil n/2 \rceil}+\lceil n/2 \rceil-k_1)$: Divide a given halved Gelfand-Tsetlin pattern with $n$ rows, bottom row~$\left( k_1, \dots, k_{\lceil n/2 \rceil} \right)$ and no entry larger than $b$ into $\nearrow$-diagonals as before and number them from left to right by $1$ to $\lceil n/2 \rceil$. For each $1 \le i \le \lceil n/2 \rceil$, add $i-k_1$ to the entries of the $i^{\text{th}}$ diagonal. Thus, we ensure that the leftmost entry in the bottom row is transformed into $1$. The entries in the Gelfand-Tsetlin pattern determine the positions of the tiles~$\uprightlozenge[.15]$. The remaining tiles are forced by these initial lozenges. Figure~\ref{figure:LozengeTiling} illustrates an example.

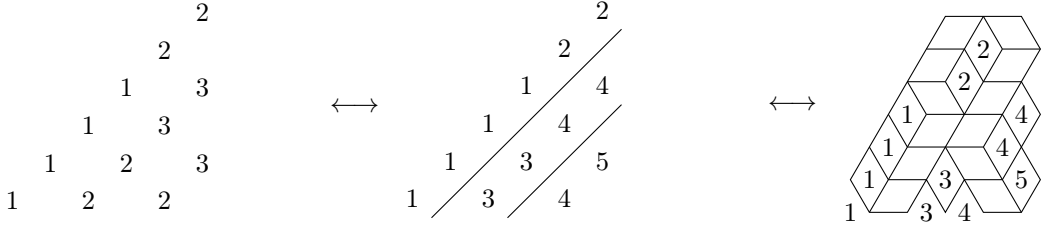
\begin{figure}[htb]
	\centering
		\raisebox{1ex}{\begin{tikzpicture}[scale=.5]
			\hmt{{1,2,2},{1,2,3},{1,3},{1,3},{2},{2}}
			\node at (9,2.5){$\longleftrightarrow$};
			\end{tikzpicture}
			\begin{tikzpicture}[scale=.5]
			\hmt{{1,3,4},{1,3,5},{1,4},{1,4},{2},{2}}
			\draw (0.5,-0.5) -- (5.5,4.5);
			\draw (2.5,-0.5) -- (5.5,2.5);
			\node at (10,2.5){$\longleftrightarrow$};
			\end{tikzpicture}}
		\begin{tikzpicture}[scale=.5]
		
		
		\coordinate (1) at (0,0) {};
		\coordinate (2) at ($ (1) + (1,0) $) {};
		\coordinate (3) at ($ (2) + (1,0) $) {};
		\coordinate (4) at ($ (3) + (1,0) $) {};
		\coordinate (5) at ($ (4) + (1,0) $) {};
		\coordinate (6) at ($ (5) + (1,0) $) {};
		
		\coordinate (7) at ($ (1) + (60:1) $) {};
		\coordinate (8) at ($ (7) + (1,0) $) {};
		\coordinate (9) at ($ (8) + (1,0) $) {};
		\coordinate (10) at ($ (9) + (1,0) $) {};
		\coordinate (11) at ($ (10) + (1,0) $) {};
		\coordinate (12) at ($ (11) + (1,0) $) {};
		
		\coordinate (13) at ($ (7) + (60:1) $) {};
		\coordinate (14) at ($ (13) + (1,0) $) {};
		\coordinate (15) at ($ (14) + (1,0) $) {};
		\coordinate (16) at ($ (15) + (1,0) $) {};
		\coordinate (17) at ($ (16) + (1,0) $) {};
		
		\coordinate (18) at ($ (13) + (60:1) $) {};
		\coordinate (19) at ($ (18) + (1,0) $) {};
		\coordinate (20) at ($ (19) + (1,0) $) {};
		\coordinate (21) at ($ (20) + (1,0) $) {};
		\coordinate (22) at ($ (21) + (1,0) $) {};
		
		\coordinate (23) at ($ (18) + (60:1) $) {};
		\coordinate (24) at ($ (23) + (1,0) $) {};
		\coordinate (25) at ($ (24) + (1,0) $) {};
		\coordinate (26) at ($ (25) + (1,0) $) {};
		
		\coordinate (27) at ($ (23) + (60:1) $) {};
		\coordinate (28) at ($ (27) + (1,0) $) {};
		\coordinate (29) at ($ (28) + (1,0) $) {};
		\coordinate (30) at ($ (29) + (1,0) $) {};
		
		\coordinate (31) at ($ (27) + (60:1) $) {};
		\coordinate (32) at ($ (31) + (1,0) $) {};
		\coordinate (33) at ($ (32) + (1,0) $) {};
		
		
		\draw (2) -- (3) -- (9) -- (4) -- (10) -- (5) -- (6) -- (12) -- (17) -- (22) -- (26) -- (30) -- (33) -- (31) -- (7) -- (2);
		
		
		\node at ($ (1) + (.5,0) $) {$1$};
		\node at ($ (3) + (.5,0) $) {$3$};
		\node at ($ (4) + (.5,0) $) {$4$};
		
		\node at ($ (7) + (.5,0) $) {$1$};
		\node at ($ (9) + (.5,0) $) {$3$};
		\node at ($ (11) + (.5,0) $) {$5$};
		
		\node at ($ (13) + (.5,0) $) {$1$};
		\node at ($ (16) + (.5,0) $) {$4$};
		
		\node at ($ (18) + (.5,0) $) {$1$};
		\node at ($ (21) + (.5,0) $) {$4$};
		
		\node at ($ (24) + (.5,0) $) {$2$};
		
		\node at ($ (28) + (.5,0) $) {$2$};
		
		
		\draw (2) -- (19);
		\draw (24) -- (32);
		
		\draw (9) -- (29);
		
		\draw (16) -- (26);
		
		\draw (11) -- (17);
		
		
		\draw (8) -- (13);
		
		\draw (14) -- (18);
		
		\draw (10) -- (15);
		\draw (19) -- (23);
		
		\draw (6) -- (16);
		\draw (20) -- (24);
		
		\draw (17) -- (21);
		\draw (25) -- (28);
		
		\draw (29) -- (32);
		
		
		\draw (8) --++(1,0);
		\draw (10) --++(1,0);
		
		\draw (14) --++(2,0);
		
		\draw (19) --++(2,0);
		
		\draw (23) --++(1,0);
		\draw (25) --++(1,0);
		
		\draw (27) --++(1,0);
		\draw (29) --++(1,0);
		\end{tikzpicture}
		\caption{Transforming a halved Gelfand-Tsetlin pattern with no entry larger than $3$ into a lozenge tiling of quartered hexagon}
		\label{figure:LozengeTiling}
\end{figure}

The lozenge tilings of the region $H_{n,b-k_1} (1 , k_2+2-k_1 , \dots , k_{\lceil n/2 \rceil}+\lceil n/2 \rceil-k_1)$ are enumerated in \cite[Theorem 3.1]{Lai14} with the numbers of these tilings being given by \eqref{eq:HMTDetEven} and \eqref{eq:HMTDetOdd}.

Regarding the interpretation in terms of representation theory, we see that halved Gelfand-Tsetlin patterns are in bijective correspondence with \emph{symplectic patterns} as defined by Proctor \cite{Pro94}: Given a halved Gelfand-Tsetlin pattern of order~$n$, bottom row~$(k_1,\dots,k_{\lceil n/2 \rceil})$ and no entry larger than $b$, replace every entry~$x$ by $b-x$ and flip the object upside down to transform it into an $n$-symplectic pattern with the partition~$(b-k_1,\dots,b-k_{\lceil n/2 \rceil})$ as top row.

Denote by $R_i$ the sum of all entries in row~$i$ -- counted from bottom to top -- of a given $n$-symplectic pattern $P$, where $R_0 \coloneqq 0$.

First, let $n$ be even. The weight of an $n$-symplectic pattern $P$ is defined as 
\begin{equation*}
	w_{\text{even}}(P)=\prod_{i=1}^{\frac{n}{2}} x_i^{R_{2i}-2R_{2i-1}+R_{2i-2}}.
\end{equation*}
Proctor showed \cite[Theorem~4.2]{Pro94} that the generating function of all $n$-symplectic patterns with top row~$\lambda$ with respect to the weight function~$w_{\text{even}}$ is given by the \emph{symplectic character}~$sp_{\lambda} (x_1,\dots,x_{n/2})$, also known as a \emph{symplectic Schur function}. It can be expressed in terms of complete homogeneous symmetric functions~$h_m (x_1,\dots,x_{n/2},x_1^{-1},\dots,x_{n/2}^{-1})$ by the following Jacobi-Trudi type formula
\begin{multline*}
sp_{\lambda} \left(x_1,\dots,x_{\frac{n}{2}}\right)\\
= \frac{1}{2} \det_{1 \le i,j \le \frac{n}{2}} \left( h_{\lambda_i -i+j} \left(x_1,\dots,x_{\frac{n}{2}},x_1^{-1},\dots,x_{\frac{n}{2}}^{-1}\right) + h_{\lambda_i -i-j+2} \left(x_1,\dots,x_{\frac{n}{2}},x_1^{-1},\dots,x_{\frac{n}{2}}^{-1}\right) \right).
\end{multline*}

Consequently, the number of all halved Gelfand-Tsetlin patterns of even order~$n$, bottom row~ $(k_1,\dots,\allowbreak k_{n/2})$ and no entry larger than $b$ is given by
\begin{multline*}
	sp_{(b-k_1,\dots,b-k_{\frac{n}{2}})} (1,\dots,1)\\
	= \prod_{1\leq i<j\leq \frac{n}{2}} \frac{(k_j-k_i+j-i)(2b+n+2-k_i-k_j-i-j)}{(j-i)(i+j)} \prod_{i=1}^{\frac{n}{2}}\frac{b+\frac{n}{2}+1-k_i-i}{i},
\end{multline*}
which follows from \cite[Exercise 24.20]{FH91}.

The classical symplectic group is only defined on even dimensional vector spaces. However, Proctor defines symplectic groups on vector spaces of odd dimension~$n$ and proves \cite[Proposition 3.1]{Pro88} that the indecomposable trace free tensor character is given by
\begin{multline*}
sp_{\lambda} \left(x_1,\dots,x_{\frac{n+1}{2}}\right)\\
= \frac{1}{2} \det_{1 \le i,j \le \frac{n+1}{2}} \left( h_{\lambda_i -i+j} \left(x_1,\dots,x_{\frac{n+1}{2}},x_1^{-1},\dots,x_{\frac{n-1}{2}}^{-1}\right) + h_{\lambda_i -i-j+2} \left(x_1,\dots,x_{\frac{n+1}{2}},x_1^{-1},\dots,x_{\frac{n-1}{2}}^{-1}\right) \right).
\end{multline*}

As in the previous case, we give a combinatorial interpretation in terms of symplectic patterns. Let $n$ be odd and define the weight of an $n$-symplectic pattern~$P$ as 
\begin{equation*}
	w_{\text{odd}}(P)=x_{\frac{n+1}{2}}^{R_n-R_{n-1}} \prod_{i=1}^{\frac{n-1}{2}} x_i^{R_{2i}-2R_{2i-1}+R_{2i-2}}.
\end{equation*}
It holds true \cite[Theorem~4.2]{Pro94} that the generating function of all $n$-symplectic patterns with top row~$\lambda$ with respect to the weight function~$w_{\text{odd}}$ is given by the symplectic character $sp_{\lambda} (x_1,\dots,x_{(n+1)/2})$. This implies
\begin{equation*}
	sp_{(b-k_1,\dots,b-k_{\frac{n+1}{2}})} (1,\dots,1)
	= \prod_{1\leq i<j\leq \frac{n+1}{2}}\frac{(k_j-k_i+j-i)(2b+n+2-k_j-k_i-j-i)}{(j-i)(j+i-1)}.
\end{equation*}

\bibliography{HoengesbergHMT}
\bibliographystyle{alpha}

\end{document}